\documentclass[12]{amsart}
 \usepackage{graphicx,amssymb,amscd,amsxtra,amsmath,amsfonts,amsthm,footnpag,
 }

\renewcommand{\thefootnote}%
{\fnsymbol{footnote}}

\theoremstyle{plain}
\newtheorem{Thm}
{Theorem}
\newtheorem{Lemm}[Thm]{Lemma}

\newtheorem*{Thm*}{Theorem}
\newtheorem*{Lemm*}{Lemma}
\newtheorem*{Propn*}{Proposition}
\newtheorem*{Cor*}{Corollary}
\newtheorem*{Res}{Resistance   Theorem}
\newtheorem*{Rado}{Rad\'o's Theorem}

\newtheorem*{Patch}{Patching Lemma}

\theoremstyle{remark}
\newtheorem*{Rmk}{\bf Remark}

\newtheorem*{Que}{\bf Question}
\newtheorem*{Defn}{\bf Definition}

\newcommand{\abs}[1]{\lvert#1\rvert}
\newcommand{\norm}[1]{\lVert#1\rVert}

\newcommand{\dliminf}{\displaystyle\liminf}

\DeclareMathSymbol
{\rightrightarrows}
{\mathrel}{AMSa}{"13}

\begin{document}
\author{Charles Pugh}
\author{Conan Wu}

\title{Jordan Curves and Funnel Sections}
\maketitle


{\hfill \today \hfill}

\section{Introduction}
\label{s.intro}
A continuous time dependent   vector ODE on $\mathbb{R}^{m}$
\begin{equation}
\label{e.ODE}
\begin{split}
y^{\prime} = f(t, y)  \qquad y(t_{0}) = y_{0} 
\end{split}
\end{equation}
can have many solutions with the same initial condition.  The simplest example is the time independent one dimensional ODE
$$
y^{\prime} = 2\abs{y}^{1/2} \quad \quad y(0) = 0
$$
whose uncountably many solutions with initial condition $y(0) = 0$ are
$$
y_{a,b}(t)  \, = \,
\begin{cases}
   -(t-a)^{2}   &  \textrm{ if } t < a
  \\
   0  &  \textrm{ if }  a \leq  t \leq b
  \\
   (t-b)^{2}  &  \textrm{ if }  b <  t
\end{cases}
$$
where $-\infty \leq a \leq  0 \leq  b \leq  \infty$.  The \textbf{solution funnel} of (\ref{e.ODE})   is
$$
F(t_{0}, y_{0}, f) = \{ (t,y(t)) : y(t)  \textrm{ solves the ODE $y^{\prime} = f(t, y)$ with $y(t_{0}) = y_{0}$}\} \ .
$$
It is the union of the graphs of the solutions with the given initial condition.  Its  cross-section at time $t_{1}$ is the \textbf{funnel section}
$$
K_{t_{1}}(t_{0}, y_{0}, f) = \{ y_{1}  : (t_{1}, y_{1})  \in F(t_{0}, y_{0}, f)\} \ .
$$
See Figure~\ref{f.section}.
\begin{figure}[htbp]
\centering
\includegraphics[scale=.65]{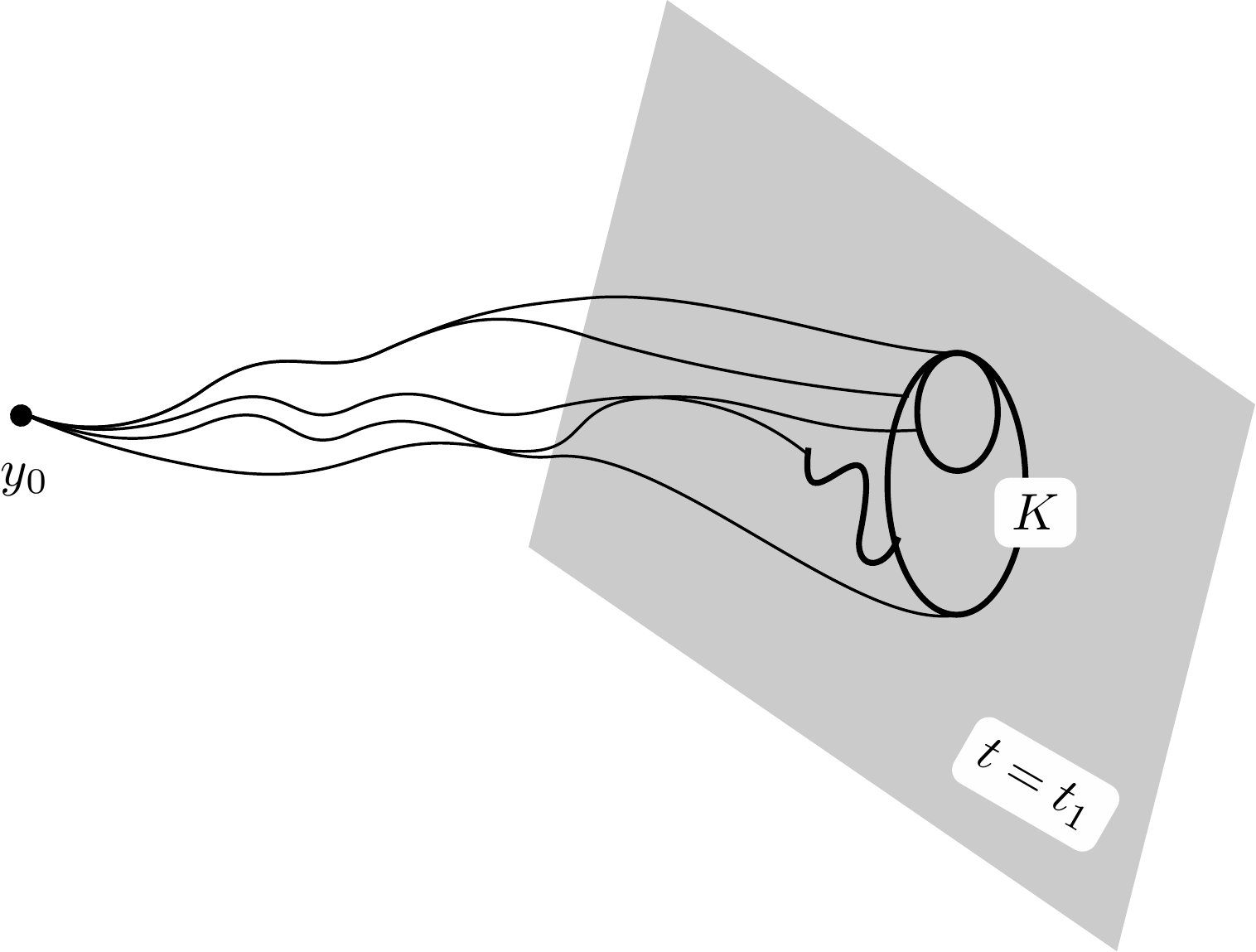}
\caption{The funnel of a two dimensional time dependent ODE   through $y_{0}$ and its cross-section $K$ in the time-$t_{1}$ plane.}
\label{f.section}
\end{figure}
The funnel section consists of the points $y_{1} \in \mathbb{R}^{m}$ that are accessible from $y_{0}$ by a solution   starting from $y_{0}$ at time $t_{0}$ and arriving at $y_{1}$ at time $t_{1}$.

The classical theorem about funnel sections  is due to H. Kneser.  It states that $K_{t_{1}}$ is a continuum (i.e., is nonempty, compact, and connected) if $f : \mathbb{R}^{m+1} \rightarrow  \mathbb{R}^{m}$   is continuous and has compact support.  See \cite{Kneser} and \cite{Pugh}.

In \cite{Pugh} the first author investigated the question: ``which continua are funnel sections?''  The main answers were
\begin{itemize}

\item[(a)]
The planar continuum consisting of an outward spiral and its limit circle is not a funnel section of any $m$-dimensional ODE.
\item[(b)]
Every continuum in $\mathbb{R}^{m}$ whose complement is diffeomorphic to $\mathbb{R}^{m} \setminus \{0\}$ is a funnel section of an $m$-dimensional ODE.
\item[(c)] All piecewise smooth, compact, connected polyhedra in $\mathbb{R}^{m}$ are funnel sections of $m$-dimensional ODEs.
\end{itemize}

(b) provides a great many pathological continua as funnel sections.  For example, all non-separating planar continua are funnel sections.  This includes the topologist's sine curve, the bucket handle, and the pseudo-arc (which  contains no ordinary arcs).  (c) implies that all compact smooth manifolds are funnel sections.

An obvious question remains open: Is the property of being a funnel section topological, or does it depend on how   the continuum is embedded in $\mathbb{R}^{m}$? The simplest case is the circle, where the question becomes: ``\emph{Is every Jordan curve a funnel section of a two dimensional ODE?}''     In this paper we answer the question affirmatively under some extra hypotheses, and point out the difficulties in general.  We also expand on some remarks in \cite{Pugh}.

\section{Smooth Pierceability}
\label{s.buttons}
An arc \textbf{pierces} a separating plane continuum $K$, such as a Jordan curve, if it   meets $K$ at a single point  and passes from one complementary component of $K$ to another. If the arc is smooth it \textbf{smoothly pierces} $K$.  Planar Jordan curves are everywhere pierceable and  smooth Jordan curves are everywhere smoothly pierceable.

\begin{Thm}
\label{t.pfunnel}
If  a planar Jordan curve is smoothly pierceable at some point then it is a funnel section of a two dimensional ODE.
\end{Thm}

\begin{Patch}
\label{l.patch}
If $K$ is a funnel section of an $m$-dimensional ODE and $p \in \mathbb{R}^{m}$ is given then there exists a continuous  $g : \mathbb{R}^{m+1} \rightarrow  \mathbb{R}^{m}$ with compact support contained in $[0,1] \times \mathbb{R}^{m}$ such that $K = K_{1}(0,p,g)$.
\end{Patch}

\begin{proof}[\bf Proof]
This is   Proposition~2.4 of \cite{Pugh}.  It lets us   patch funnels together, one to the next.    
\end{proof}

\begin{proof}[\bf Proof of Theorem~\ref{t.pfunnel}.]  
Let $J \subset  \mathbb{R}^2$ be a Jordan curve   pierced at $p$ by a smooth  arc $A$.    We may assume $p$ is the origin and $A$ contains the horizontal segment $[-1,1] \times \{0\}$.  Let $\beta  : \mathbb{R} \rightarrow  [0,1)$ be a smooth bump function such that the eye-shaped region $S$ between the graphs of $-\beta $ and $\beta $ is as in Figure~\ref{f.eye-shaped region}.  
\begin{figure}[htbp]
\centering
\includegraphics[scale=.55]{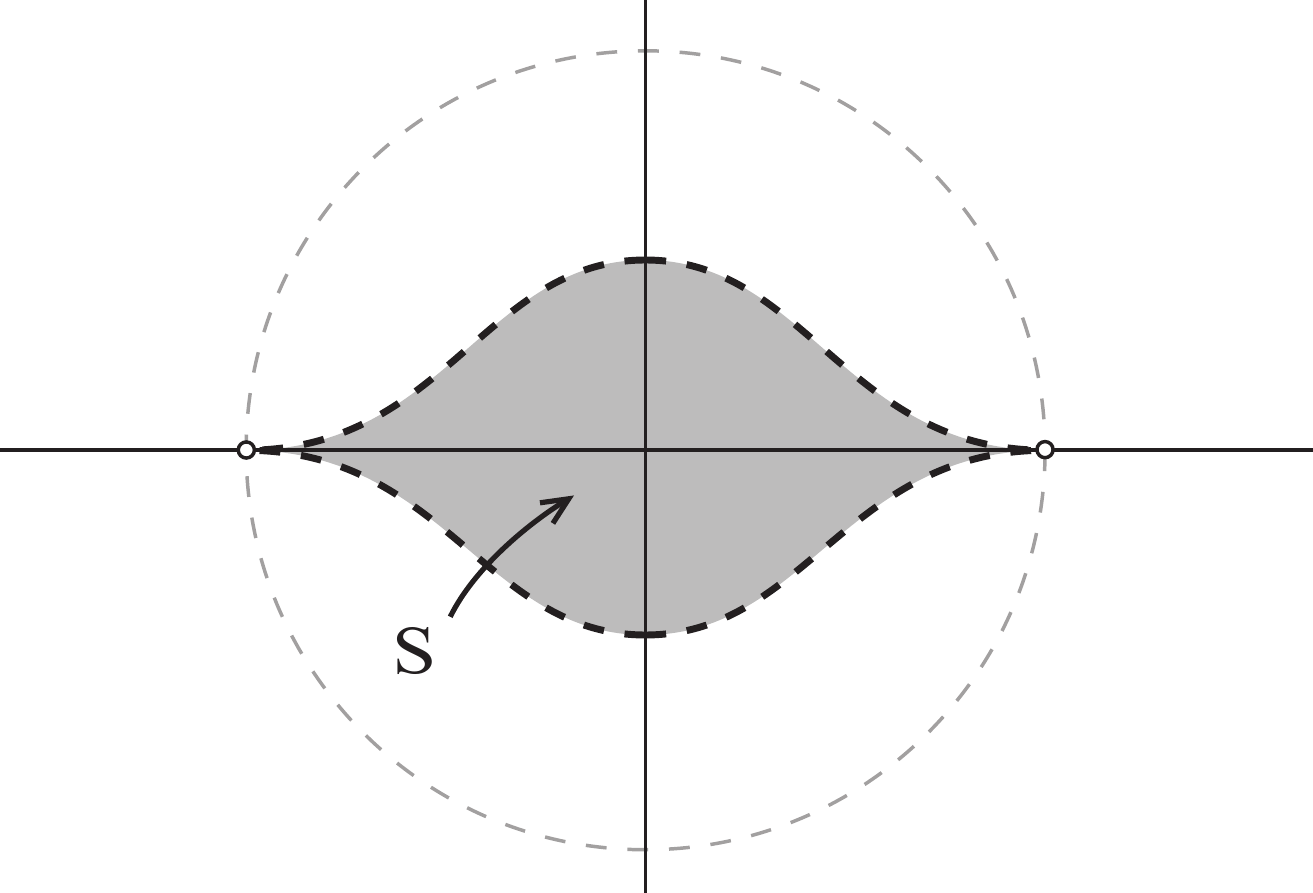}
\caption{The eye-shaped region $S$.}
\label{f.eye-shaped region}
\end{figure}
There is a smooth map $\Psi  : \mathbb{R}^{2} \setminus  S \rightarrow  \mathbb{R}^{2}$ closing the eye.  It sends the vertical segment $x \times [\beta (x),1]$ to $x \times  [0,1]$, is symmetric with respect to $y \mapsto  -y$,  and is the identity outside the unit disc $\mathbb{D} $.  Except for the fact that $\Psi (x,\beta (x)) = (x,0) = \Psi (x,-\beta (x))$, $\Psi $ is a diffeomorphism.  

The map $\Psi ^{-1} : \mathbb{R}^{2} \setminus (-1,1) \times  0$ opens the eye and sends $J \setminus 0$ to a an open arc.  Let $J^{\displaystyle *} = \Psi^{-1}(J\setminus 0) \cup (0, \beta (0)) \cup (0, -\beta (0))$.   It is a compact planar arc, so it is a funnel section: $J^{\displaystyle *} = K_{1}(0,p,f)$ for   some $p \in \mathbb{R}^{2}$ and some $f = f(t,x,y)$ with compact support in $[0,1] \times  \mathbb{R}^{2}$. (Here we used the Patching Lemma and the fact from \cite{Pugh} that every planar continuum whose complement is diffeomorphic to $\mathbb{R}^{2} \setminus  0$ is a funnel section.)

Next, we gradually close the eye as $t$ varies from $t=1$ to $t=2$.  There is a continuous  vector field $g= g(t,x,y)$ tangent to the vertical lines $x \times  \mathbb{R}$ whose forward trajectories on $[1,2] \times x\times [-1,1]$ are shown  in Figure~\ref{f.Psiflow}.    
\begin{figure}[htbp]
\centering
\includegraphics[scale=.55]{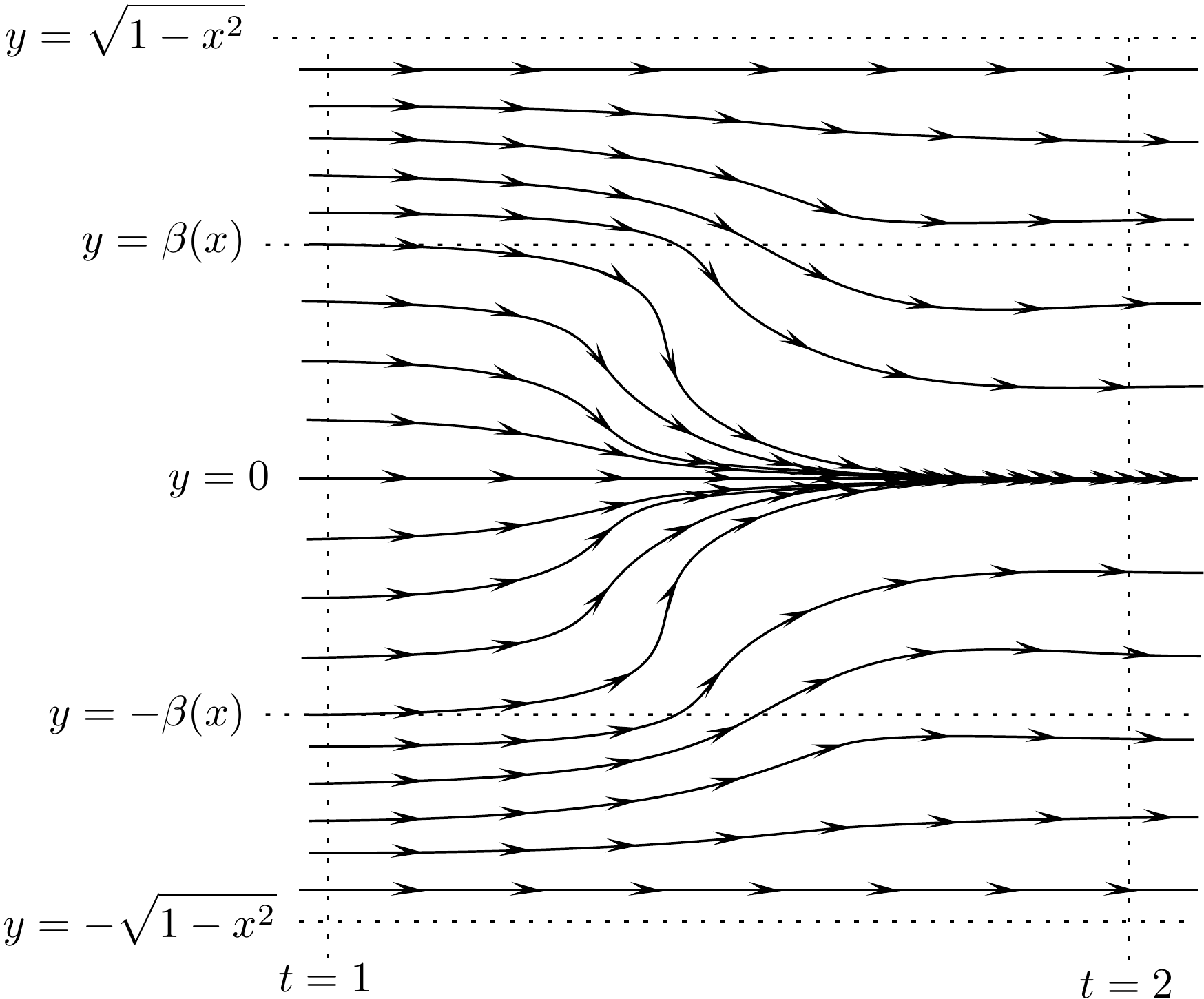}
\caption{A $g$-trajectory starting at $(1,x,y)$ with $\abs{y} \geq  \beta (x)$ ends at $(2,\Psi (x,y))$.}
\label{f.Psiflow}
\end{figure}
The time-one map of the forward $g$-flow is $\Psi $.  See Figure~\ref{f.Psi}.  
  \begin{figure}[htbp]
\centering
\includegraphics[scale=.7]{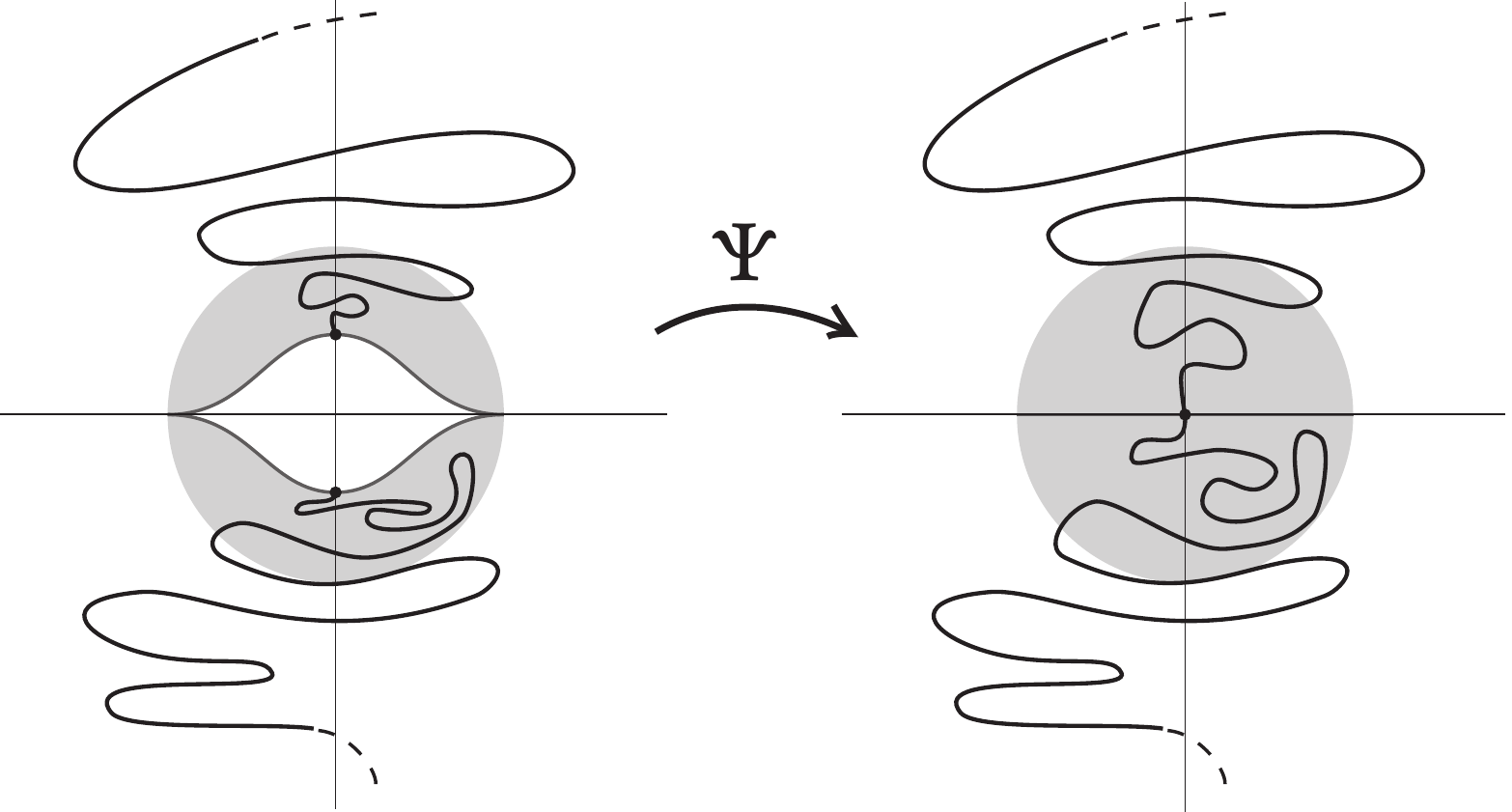}
\caption{$\Psi$ closes the eye.}
\label{f.Psi}
\end{figure}
Thus  $J$ is a funnel section, $J = K_{2}(0,p,f+g)$. 
\end{proof}

\begin{Thm}
\label{t.pJordan}
There exist planar Jordan curves, smoothly pierceable at no points. Some  of them are funnel sections.
\end{Thm}

See Section~\ref{s.diffeotopies} for the proof of the second assertion.

\begin{Rmk}
It is natural to expect that   unions and intersections of funnel sections of $m$-dimensional ODEs are funnel sections of $m$-dimensional ODEs. The union case is an open question, while the intersection assertion is false.  In fact, if it were known that the union of two funnel sections of $2$-dimensional ODEs   is a funnel section of a $2$-dimensional ODE then it would follow at once that every planar Jordan curve $J$ is a funnel section of a $2$-dimensional ODE.  For $J$ is the union of two arcs, each being a funnel section of a $2$-dimensional ODE by (b) in Section~\ref{s.intro}.  See Section~\ref{s.oneup} for the union question when it is permitted to raise the dimension.

To understand the funnel intersection question, consider the   outward spiral together with its limit circle, $\overline{S}$.  It is not a funnel section, but its one-point suspension $T$ is one.  For the complement of $T$ in $\mathbb{R}^{3}$ is diffeomorphic to the complement of a point.  The closed unit disc $\overline{\mathbb{D} } \subset \mathbb{R}^{2} \subset \mathbb{R}^{3}$ is a funnel section, but $\overline{S} = T \cap \overline{\mathbb{D} }$ is not.
\end{Rmk}

\section{Nowhere smoothly Pierceable Jordan curves}
\label{s.nowhere}
It is not surprising that there exist planar Jordan curves which are nowhere smoothly pierceable.  We prove slightly more.

\begin{Thm}
\label{t.pJordanL}
There are planar Jordan curves that are nowhere pierceable by paths of finite length.  Some of them are funnel sections.
\end{Thm}

See Section~\ref{s.diffeotopies} for the proof of the second assertion.

If $J \subset S^{2}$ is a Jordan curve that separates the north and south poles and $\gamma $ is a path from one pole to the other that pierces   $J$   then we call $\gamma $ a \textbf{polar path}  for $J$.  Every polar path has length $\geq  \pi $.  The \textbf{resistance} of $J$ is
$$
r(J) = \inf \{ \ell(\gamma ) : \gamma  \textrm{ is a polar path for } J \}
$$
where $\ell (\gamma )$ is the length of $\gamma $.   

The greater the resistance, the longer it takes a     point to travel at unit speed from pole to pole crossing $J$ just once.  The equator has resistance $\pi $.  Approximating it by a  Jordan curve $B$ made of many small consecutive buttons as in Figure~\ref{f.button} does not increase the resistance, but subsequently adding zippers across the buttons as in Figure~\ref{f.zipperedbutton} increases the resistance as much as we want.
\begin{figure}[htbp]
\centering
\includegraphics[scale=.55]{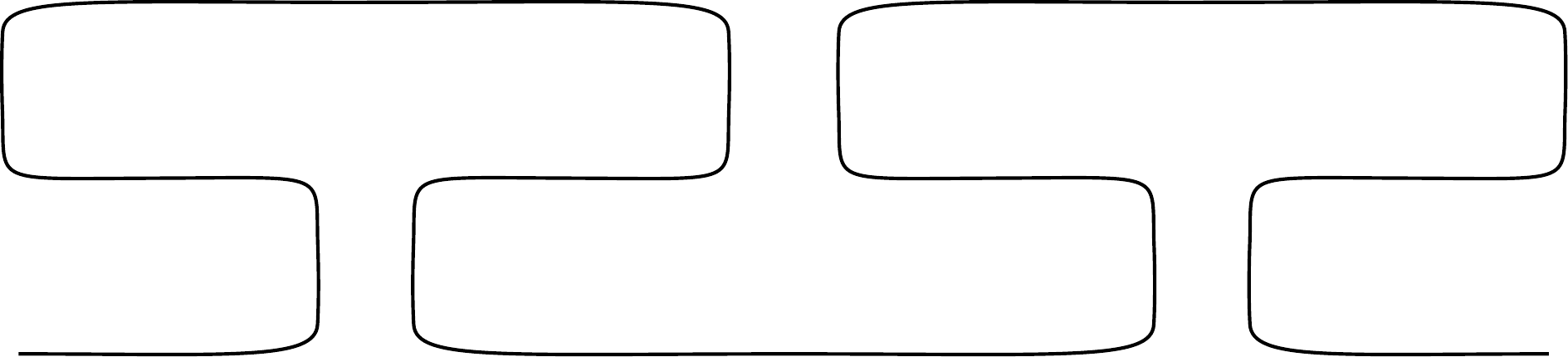}
\caption{Part of a button curve $B$.}
\label{f.button}
\end{figure}

\begin{figure}[htbp]
\centering
\includegraphics[scale=.55]{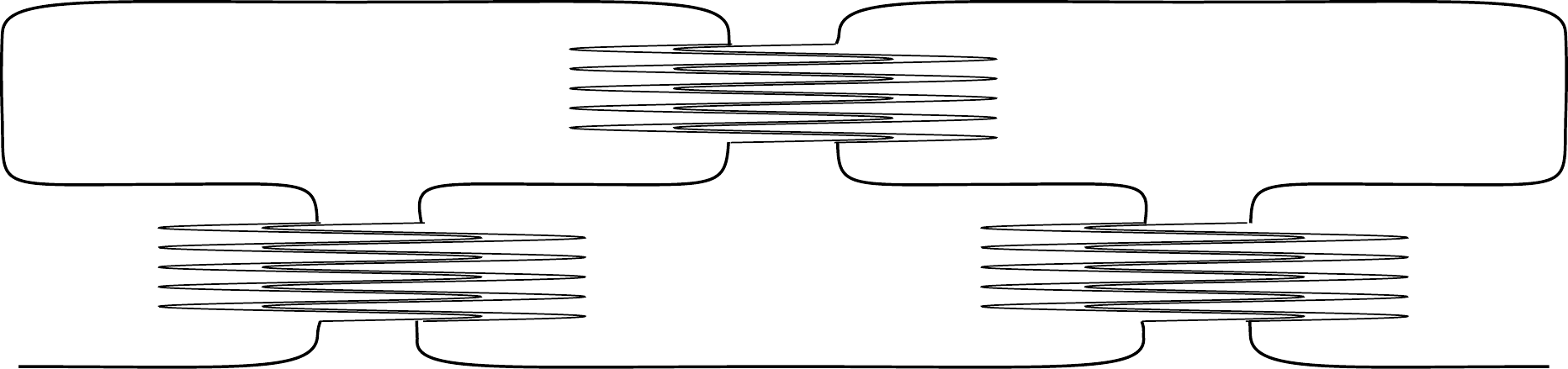}
\caption{Part of a zippered button curve  --  a \textbf{resistor curve}.}
\label{f.zipperedbutton}
\end{figure}

\begin{Res}
\label{t.resistance}
There exist Jordan curves with infinite resistance.
\end{Res}

\begin{proof}[\bf Proof.]
Modify the equator by approximating it with a smooth resistor curve  $R_{1}$ having resistance $> \pi $.  Then   approximate $R_{1}$ with a resistor curve $R_{2}$ having resistance $> 2\pi $, etc., and take a limit.  The details appear below.
\end{proof}

\begin{Lemm}
\label{l.strip}
There is a diffeomorphism $\sigma: [0,1]^{2} \rightarrow [0,1]^{2}$ such that $\sigma   $ is the identity on a neighborhood of the boundary and $\sigma   $ carries the rectangle $[2/5, 3/5] \times [0,1]$ to an S-shaped strip  $S$  such that for every path $\gamma (t) = (x(t), y(t))$ in $S$   connecting the top and bottom of $S$, $0 \leq  t \leq  1$, we have
$$
\int_{0}^{1} \abs{x^{\prime}(t)} \, dt  \geq  1 \ .
$$
\end{Lemm}

\begin{proof}[\bf Proof.]
See Figure~\ref{f.newS}.
\begin{figure}[htbp]
\centering
\includegraphics[scale=.55]{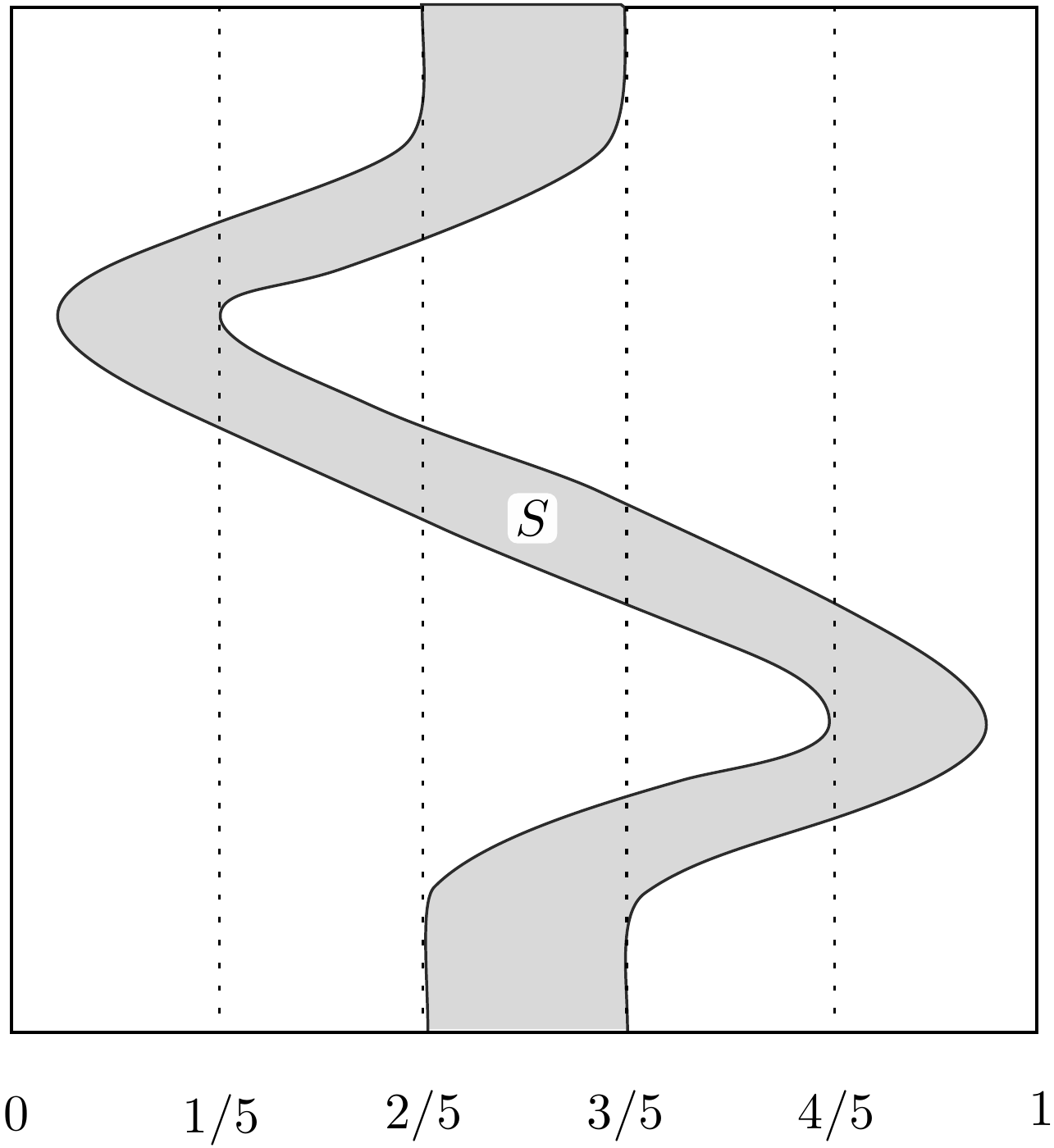}
\caption{The S-shaped strip $S$.}
\label{f.newS}
\end{figure}
\end{proof}

\begin{Lemm}
\label{l.zipper} 
Given $L > 0$, there  is a diffeomorphism $\zeta: [0,a] \times  [0,b] \rightarrow [0,a] \times [0,b]$ such that $\zeta  $ is the identity on a neighborhood of the boundary and $\zeta  $ carries the rectangle $[2a/5, 3a/5] \times [0,b]$ to a zipper  strip  $Z$  such that every path $\gamma  $ in $Z$   connecting the top and bottom of $Z$  has
length $\geq L$.
\end{Lemm}

\begin{proof}[\bf Proof.]
Reduce the unit square to a rectangle $[0,a] \times  [0,b/n]$ with $n \geq  L/a$.  Stack $n$ copies of the reduced strip diffeomorphism $\sigma $  from Lemma~\ref{l.strip} to form  $\zeta $ and $Z$.  The length of $\gamma $ is at least $na \geq  L$. See Figure~\ref{f.newZ} in which $a=1$, $b=1/2$, and $n=4$.
\end{proof}

\begin{figure}[htbp]
\centering
\includegraphics[scale=.55]{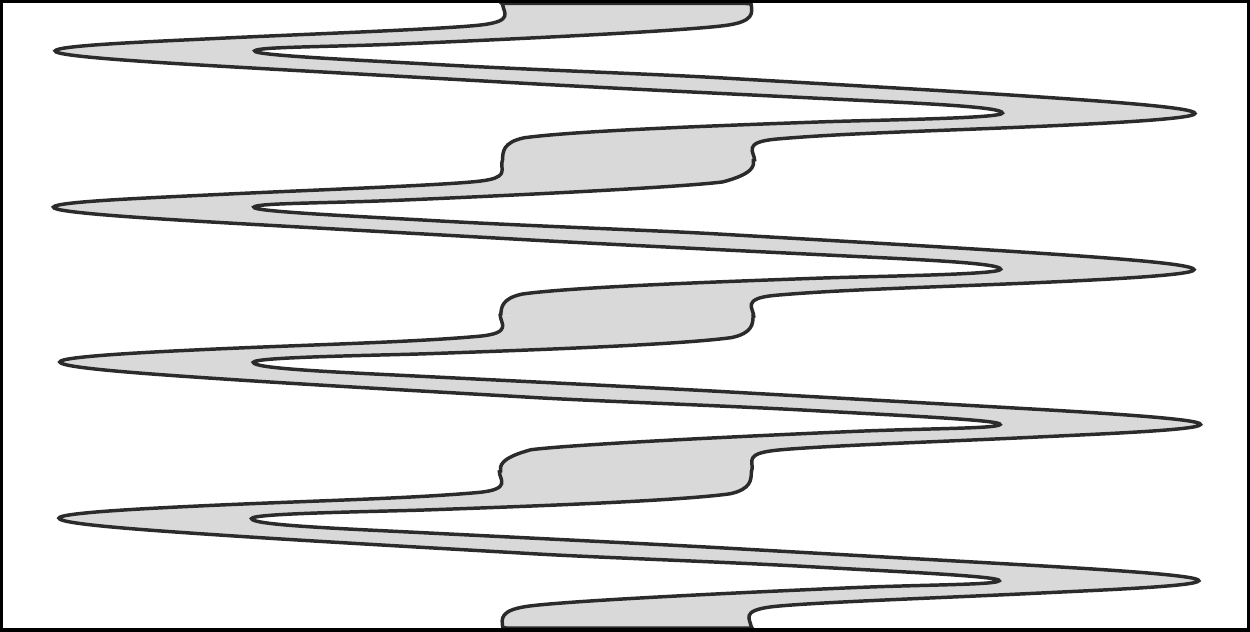}
\caption{A stack of four unit-length S-strips forms a zipper strip $Z$ with    $r(Z) \geq  4$.}
\label{f.newZ}
\end{figure}

 Let $C$ be the cylinder $S^{1} \times  [-1,1]$ and let $J$ be a Jordan curve that separates its top and bottom. As for paths on the sphere, a path on the cylinder  connecting the   top and bottom   is polar for   $J$ if  it meets $J$ exactly once, and the  resistance of $J$ is the infimum of the lengths of the polar paths.  Let $E$ be the equator of $C$.

\begin{Lemm}
\label{l.zippers}
Given $\epsilon  > 0$ and $L > 0$, there is a diffeomorphism $\varphi  : C \rightarrow  C$ in the $\epsilon $-neighborhood of the identity such that $\varphi   $ is the identity off the      $\epsilon $-neighborhood of $E$,  and   $J = \varphi  (E)$  has resistance $> L$.
\end{Lemm}

\begin{proof}[\bf Proof.] Choose $a =1/n \leq \epsilon /2$ and divide    the cylinder $C_{a} = S^{1} \times [-a, a]$ into $n$ squares   of size $a \times  a$.  Draw a button curve  $B $ with $n$ buttons, one in each square.  There is a diffeomorphism $\phi    : C  \rightarrow  C $ sending each square to itself such that $\phi   (E) = B $ and $\phi $ is the identity off $C_{a}$.

Then draw  $2n$ rectangles $\rho, \rho ^{\prime} $ of length $a/2$ and height $b$ as shown in Figure~\ref{f.zippers}.
\begin{figure}[htbp]
\centering
\includegraphics[scale=.55]{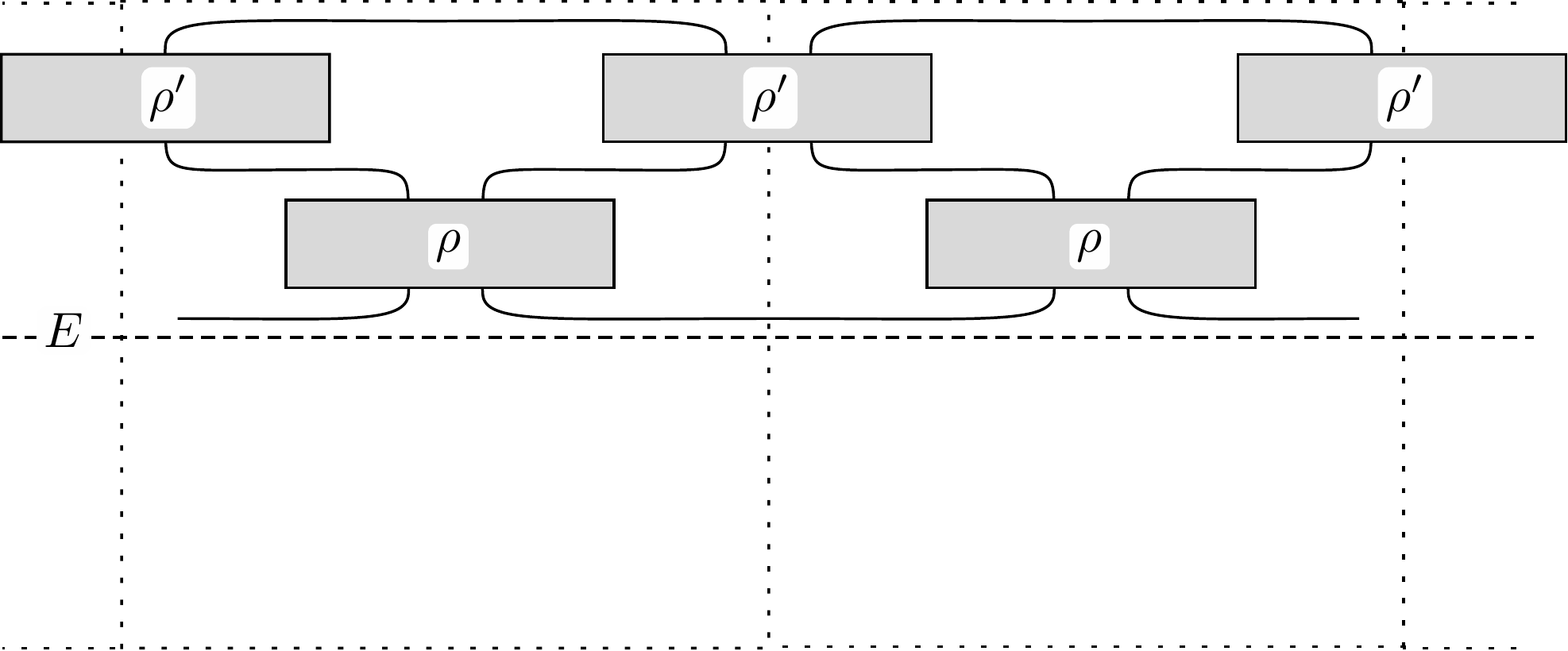}
\caption{Zipper strips are glued into the rectangles $\rho , \rho ^{\prime} \subset  C_{a}$.}
\label{f.zippers}
\end{figure}
In each   rectangle, replace the identity map by the zipper diffeomorphism constructed in Lemma~\ref{l.zipper}.  This gives a diffeomorphism $\zeta   : C \rightarrow  C$.  The composite $\varphi = \zeta   \circ  \phi $ $\epsilon $\nobreakdash-approximates the identity and fixes all points off $C_{a}$.   Every polar path for  $J = \varphi (E)$  must travel through an entire zipper strip and therefore has length $> L$.
\end{proof}

Let $\mathcal{J}_{0}$ be the collection of Jordan curves on the 2-sphere that separate the poles.

\begin{Lemm}
\label{l.lsc}
With respect to the Hausdorff metric, the
  resistance function is lower semi-continuous at smooth Jordan curves in $\mathcal{J}_{0}$.
\end{Lemm}

\begin{proof}[\bf Proof.]
Let   $J_{n}$ be a sequence of Jordan curves in $\mathcal{J}_{0}$ that converges to $J \in \mathcal{J}_{0}$ with respect to  the Hausdorff metric.  If $J$ is smooth we claim that $r(J) \leq  \dliminf_{n\rightarrow \infty} r(J_{n})$.  We refer to points on the south side of a curve in   $  \mathcal{J}_{0}$ as ``below''   and those on the north side as ``above.''

Consider the $\epsilon $-tubular neighborhood $N = N_{\epsilon }$ of $J$.  It is an annulus bounded by smooth Jordan curves $J_{a}$ and $J_{b}$ above and below $J$.    The normals to $J$ give smooth projections $\pi _{a} : N  \rightarrow J_{a}$, $\pi _{b} : N \rightarrow  J_{b}$.  As $\epsilon \rightarrow 0$, the norms of $(D\pi _{a})_{x}$ and $(D\pi _{b})_{x}$ for $x \in N $ tend  uniformly to $1$.  Thus, if  $\nu  $ is a smooth path in $N $ then the length ratio satisfies
$$
\liminf_{\epsilon \rightarrow 0} \frac{\ell(\nu  )}{\ell(\pi _{a}(\nu  ))} \geq  1 \qquad \liminf_{\epsilon \rightarrow 0} \frac{\ell(\nu  )}{\ell(\pi _{b}(\nu  ))} \geq  1 \ ,
$$
uniformly  $\nu  \subset N $.

Fix a small $\epsilon  > 0$ and    choose a polar path $\gamma _{n}$ for $J_{n}$ whose length is approximately $r(J_{n})$.  For large $n$, $J_{n} \subset  N $ and
$J_{n}$ separates the boundary curves  $J_{a}, J_{b}$ of $N $.
By approximation, we can assume that $\gamma _{n}$ is smooth except at the point $p_{n} \in J_{n}$ where it crosses $J_{n}$, and that $\gamma _{n}$ is transverse to $\partial N $.    We form a polar path $\rho _{n} $ for $J$ as follows.  (It will not be much longer than $ \gamma _{n}$.)

The polar path $\gamma _{n}$ for $J_{n}$ goes from the north pole to the south pole, and $J_{n}$ splits it as $\gamma _{n} = \alpha   \cup \beta  $ where $\alpha    $ goes from the north pole to $p_{n}$ and $\beta    $ goes from $p_{n}$ to the south pole.  Since $J_{n}$ separates the boundary curves of $N $,  $\alpha  $ lies above $J_{b}$ and $\beta  $ lies below $J_{a}$.  Transversality implies that $\alpha  $ splits as
$$
\alpha   = \alpha  _{1} \cup \nu  _{1} \cup   \dots \cup \alpha _{k} \cup\nu  _{k}
$$
where each $\alpha  _{j}$ lies above $J_{a}$ and each $\nu  _{j}$ lies in $N $.  The curve
$$
\rho _{a} = \alpha  _{1} \cup \pi _{a}(\nu  _{1}) \cup   \dots \cup \alpha _{k} \cup \pi _{a}(\nu  _{k})
$$
lies above or on $J_{a}$ and has length not much greater than $\ell(\alpha ) $. (In fact, it is likely that $\ell(\rho _{a})$ is much less than $\ell(\alpha )$.)  See Figure~\ref{f.lsc}.
\begin{figure}[htbp]
\centering
\includegraphics[scale=.55]{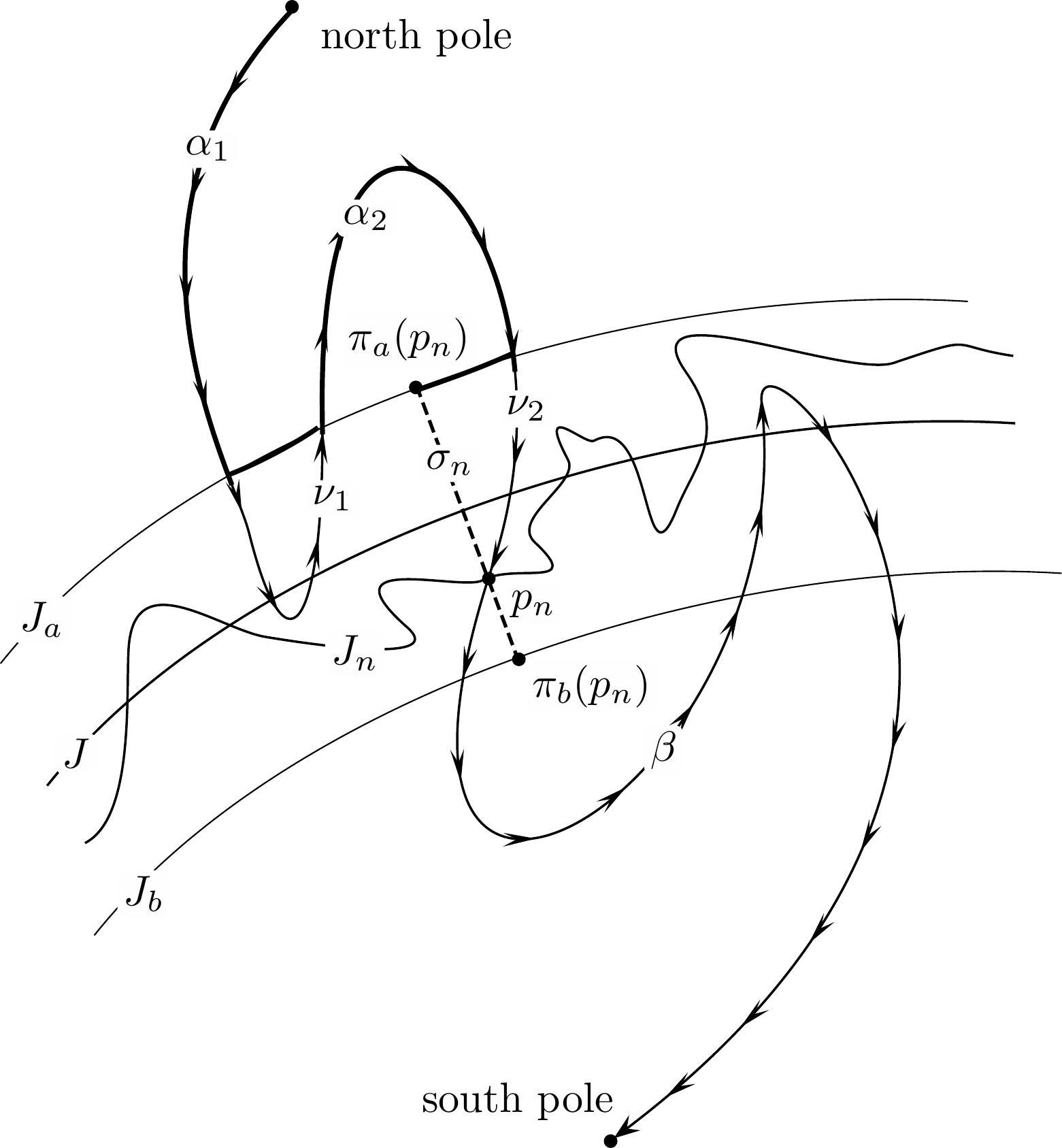}
\caption{$J_{a}$ splits $\alpha $ as $ \alpha  = \alpha _{1} \cup  \nu_{1} \cup \alpha _{2} \cup \nu _{2}$.  The path $\rho _{a}$ is drawn thick.}
\label{f.lsc}
\end{figure}
In the same way we form from $\beta  $ a path $\rho _{b}$ that lies below or on $J_{b}$ and has length not much greater than $\ell (\beta  )$.  The path $\rho _{a}$ ends at the point $\pi _{a}(p_{n})$ while $\rho _{b}$ starts at the point $\pi _{b}(p_{n})$.  Let $\sigma _{n} = [\pi _{a}(p_{n}), \pi _{b}(p_{n})]$ be the normal segment of $N $ that passes through $p_{n}$.   Thus,
$$
\rho _{n} = \rho _{a} \cup \sigma _{n} \cup \rho _{b}
$$
is a polar path  $J$ and its length is not much greater than $\ell (\gamma _{n})$.  It follows that $r(J) \leq \dliminf _{n\rightarrow \infty}r(J_{n})$.
\end{proof}

\begin{Rmk}
The resistance function is not upper semicontinuous.  There exist resistor   curves approximating the equator arbitrarily well that have large resistance.
\end{Rmk}

 \begin{Que}
Is the preceding lemma true without the assumption that $J$ is smooth?  That is, if   $J_{n} \rightarrow J$ in $\mathcal{J}_{0}$ and there are polar paths $\gamma _{n}$ for $J_{n}$ of length $\leq  r$, is there a polar path for $J$ of length $\leq  r + \epsilon $?
\end{Que}

Lemma~\ref{l.lsc} uses the Hausdorff metric on $\mathcal{J}_{0}$.  A finer topology is defined as follows.  Every parameterization of a Jordan curve is an embedding $f : S^{1} \rightarrow  S^{2}$, and every $f$ extends to a homeomorphism $F : S^{2} \rightarrow  S^{2}$.  (We think of the circle as the equator of the sphere.)  The space $\mathcal{H}$ of self-homeomorphisms of the sphere has a natural metric
$$
D(F,G) = \norm{F-G} + \norm{F^{-1} - G^{-1}} \ ,
$$
where $\norm{F_{1}-F_{2}} = \sup \{ \abs{F_{1}(x) - F_{2}(x)} : x \in S^{2}\}$ is   $C^{0}$-distance. 
With respect to $D$, $\mathcal{H}$ is complete, and the subset
$$
\mathcal{H}_{0} = \{ F \in \mathcal{H} : F(S^{1}) \textrm{ separates the poles}\}
$$
is closed in $\mathcal{H}$.

\begin{Lemm}
\label{l.generic}
For the generic $F \in \mathcal{H}_{0}$,   $  F(S^{1})$    has infinite resistance.
\end{Lemm}

\begin{proof}[\bf Proof.]
It suffices to check that for every $L > 0$,
$$
\mathcal{H}_{0}(L) = \{ F \in \mathcal{H} :   F(S^{1}) \textrm{ has resistance $> L$} \}
$$
contains an open dense subset.  Let $F_{0} \in \mathcal{H}_{0}$ be given.  It can be approximated in $\mathcal{H}_{0}$ by a diffeomorphism $F_{1}$.  The tubular neighborhood of the smooth Jordan curve $J_{1} = F_{1}(S^{1})$ is diffeomorphic to the cylinder, so Lemma~\ref{l.zippers} provides a diffeomorphism $F_{2}$ that  approximates the identity  and   $J = F_{2}(J_{1})$ has resistance $> L$.  Then $F = F_{2} \circ  F_{1}$ approximates $F_{1}$ and lies in $\mathcal{H}_{0}(L)$.  Hence $\mathcal{H}_{0}(L) \cap C^{\infty}$ is dense in $\mathcal{H}_{0}$.  For each $F \in \mathcal{H}_{0}(L) \cap C^{\infty}$,   $J = F(S^{1})$ is smooth, so    Lemma~\ref{l.lsc}   implies that for all $G \in \mathcal{H}_{0}$ near $F$, $G(S^{1})$ has resistance $> L$.  That is, $\mathcal{H}_{0}(L)$ contains a neighborhood of $F$.  Hence $\mathcal{H}_{0}(L)$ contains an open dense subset of $\mathcal{H}_{0}$, and $\bigcap_{L\in \mathbb{N}} \mathcal{H}_{0}(L)$ is residual; that is, for the generic $F \in \mathcal{H}_{0}$, $J = F(S^{1})$ has infinite resistance.
\end{proof}

\begin{proof}[\bf Proof of the Resistance Theorem.]
Since residual subsets of a complete nonempty metric space are nonempty, Lemma~\ref{l.generic} provides many Jordan curves of infinite resistance.
\end{proof}

\begin{Rmk}
A Jordan curve $J$ of infinite resistance is nowhere  smoothly pierceable.  For if $\nu  $ is   a smooth path piercing $J$  then we can choose a smooth path  $ \alpha  $ from    the north pole to one endpoint of $\nu  $, and a smooth path $\beta  $ from the other endpoint of $\nu  $ to the south pole, such that $\alpha  $ and $\beta  $ are disjoint from $J$.  Then the combined path $\alpha  \cup \nu   \cup \beta  $ is polar with finite length, contradicting $r(J) = \infty$.
\end{Rmk}

\begin{Rmk}
There is nothing special about the poles of the sphere.  For any distinct $p, q \in S^{2}$ we can consider the set $\mathcal{H}_{pq}$ of homeomorphisms $F \in \mathcal{H}$ such that $F(S^{1})$ separates $p$ from $q$.  Letting $p, q$ vary in a countable dense subset of the sphere, we infer a stronger looking version of the Resistance Theorem.
\end{Rmk}

\begin{Thm}
\label{t. }
For the generic $F \in \mathcal{H}$, the Jordan curve $J = F(S^{1})$ offers infinite resistance to all paths piercing it.   
\end{Thm}

\section{One Dimension Up}
\label{s.oneup}

The outward spiral together with its limit circle, $\overline{S}$, is  not a funnel section of any continuous $2$-dimensional ODE, and in fact it is not a funnel section of any continuous $m$-dimensional ODE.  Raising the permitted dimension has no effect on this property of $\overline{S}$.   However, some funnel questions  get easier if the dimension can be increased.

\begin{Thm}
\label{t.proj}

The image of a funnel section under projection is a funnel section.

\end{Thm}

\begin{proof} 
There is a continuous function $g(t,x)$ on $\mathbb{R}^{2}$ such that the trajectories of $x^{\prime} = g$ are as in Figure~\ref{f.Psiflow}: All trajectories $x(t)$ that begin in the interval $[-1,1]$ at time $t=1$ end at $0$ by time $t=2$.  
The support of $g$ is compact and contained in $[1,2] \times \mathbb{R}$.  This gives  a local projection.  

Suppose that $K = K_{1}(0,p,f)$ is a subset of the unit $m$-cube $Q$ for some continuous $f = f(t, x^{1}, \dots  , x^{m})$ with compact support in $[0,1] \times  \mathbb{R}^{m}$.  Let $\pi  : \mathbb{R}^{m} \rightarrow  \mathbb{R}^{m-1}$ be the projection that kills the span of the last variable $x^{m}$.  Then 
$$
\pi (K) = K_{2}(0,p,f+g)
$$
where $g$ is the vector field on $\mathbb{R}^{m}$,
$$
(0, \dots , 0, g(t, x^{m})) = g(t, x^{m}) \frac{\partial }{\partial x^{m}} \ .
$$
Projections into higher codimension subspaces are handled by induction.
\end{proof}

\begin{Cor*}

Every planar Jordan curve is a funnel section of $3$-dimensional ODE.

\end{Cor*}

\begin{proof} Let $h : [0,2\pi ] \rightarrow  J$ parametrize the Jordan curve $J \subset  \mathbb{R}^{2}$, and define $g : [0,2\pi ] \rightarrow  \mathbb{R}^{3}$ by
$$
g(\theta ) = (h(\theta ), \theta ) \ .
$$
$g([0,2\pi ])$ is an arc $K$ in $\mathbb{R}^{3}$.  Its complement is diffeomorphic to the complement of a point, so it is a funnel section of a $3$-dimensional ODE.  
Theorem~\ref{t.proj} implies that $\pi (K) = J$ is a funnel section of a $3$-dimensional ODE.
\end{proof}

In fact, we have established something a bit more general.

\begin{Thm*}
Peano continua in $\mathbb{R}^{m}$ are funnel sections in one dimension up.
\end{Thm*}

\begin{proof}[\bf Proof]
A Peano continuum $X$ is   the continuous image of an interval.  (Equivalently, by the Hahn-Mazurkiewicz Theorem a Peano continuum is a compact Hausdorff space which is   connected and locally connected.)   Jordan curves are Peano continua.  If $h : [0,2\pi ] \rightarrow  X \subset \mathbb{R}^{m}$ is a continuous surjection then $\theta  \mapsto g (\theta ) = (h(\theta ), \theta )$ is a homeomorphism from the interval to an arc $K \subset  \mathbb{R}^{m+1}$.  The latter is a funnel section of an $(m+1)$-dimensional ODE, and by Theorem~\ref{t.proj},  so is $X = \pi (K)$.
\end{proof}

\begin{Cor*}
The Hawaiian earring is a funnel section of a $3$-dimensional ODE.
\end{Cor*}
\begin{proof}[\bf Proof]
The Hawaiian earring is a planar Peano continuum. 
\end{proof}

\begin{Rmk}
It is not hard to show directly that the Hawaiian earring is also a funnel section of a $2$-dimensional ODE.  
\end{Rmk}

\begin{Thm}
\label{t.union}

If a continuum is a union of two funnel sections then it is   a funnel section in one dimension up.

\end{Thm}

\begin{proof} Suppose that $A, B \subset  \mathbb{R}^{m}$ are funnel sections for $m$-dimensional ODEs, $A = K_{1}(0,p,f)$ and $B = K_{1}(0,q,g)$ for continuous $f,   g : \mathbb{R}^{m+1} \rightarrow  \mathbb{R}^{m}$ having compact support in $[0,1] \times  \mathbb{R}^{m}$, and $c \in A \cap B$.  It suffices to construct  a funnel section $K$  consisting of a line segment $L = c \times [0,3]$ and   copies of $A$, $B$ as shown in Figure~\ref{f.hsection}.
\begin{figure}[htbp]
\centering
\includegraphics[scale=.6]{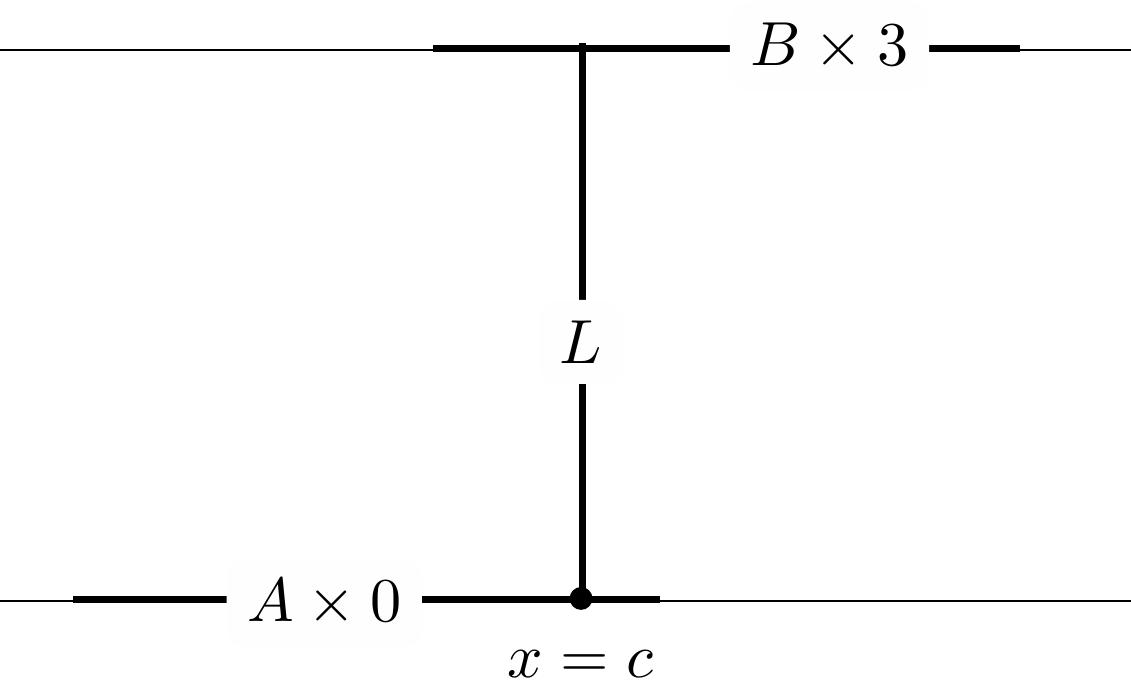}
\caption{The configuration of the desired funnel section is  $K = A \times  0 \; \cup \; L \; \cup \; B \times  3$.}
\label{f.hsection}
\end{figure}
For then Theorem~\ref{t.proj} implies   $\pi (K) = A\cup B$ is a funnel section of an $(m+1)$-dimensional ODE.

Without loss of generality we assume that the interior of the unit cube $Q = Q^{m+1}$ contains the supports of $f$, $g$,  and the funnels through $p$ and $q$.  

We write $(t,x,z) \in \mathbb{R} \times  \mathbb{R}^{m} \times  \mathbb{R}$ systematically.  It is easy to construct a continuous $h_{0} : \mathbb{R}^{m+2} \rightarrow  \mathbb{R}^{m+1}$ with compact support in $[-1,0] \times  \mathbb{R}^{m+1}$ such that
\begin{equation*}
\begin{split}
K_{0}(-1,p,h_{0}) &= L_{0} 
\\
&= [(0,p,0), (0,p, 1)] \cup [(0,p,1), (0,q, 2)] \cup [(0,q,2), (0,q, 3)]  \ .
\end{split}
\end{equation*}
$L_{0}$ is the broken line in the $t=0$ plane from $(0,p,0)$ to $(0,q,3)$ having vertices $(0,p,1)$ and $(0,q,2)$.  Then we will construct $h  $ so that $K_{1}(0,L_{0},h) = K$.  This gives $K_{1}(-1,p,h_{0}+h ) = K$.

First we fix $f$- and $g$-solutions $a(t)$ and $b(t)$ such that $a(0) = p$, $b(0) = q$, and $a(1) = c = b(1)$.  Then we construct $h $ on the three slabs $0 \leq z \leq 1$, $1 \leq  z \leq  2$, $2 \leq  z \leq  3$ as follows.  We think of $z$ as an ``external homotopy variable'' by requiring that the $z$-component of $h $ is identically zero.  This forces $h $-solutions to stay in $z=const.$ planes. For clarity we drop the zero $z$-component from the notation for $h $ and write $h (t,x,z)$ as an $m$-vector. Choose a bump function $\beta (t,x,z)$ on the bottom slab $0 \leq  z \leq  1$ such that 
\begin{itemize}

\item[(i)]
$\beta  = 1$ on the set $\{(t, x, z) \in  Q \times (0,1] : \abs{x-a(t)} \leq  2z\}$.
\item[(ii)]
$\beta =0$ on the set $\{(t, x, z) \in Q \times [0,1] : \abs{x-a(t)} \geq  4z\}$.
\item[(iii)]
$0 \leq  \beta \leq  1$ otherwise.
\end{itemize}
See Figure~\ref{f.betaa}, and note that $\beta  =1$ on the set $Q \times  1$.  
\begin{figure}[htbp]
\centering
\includegraphics[scale=.65]{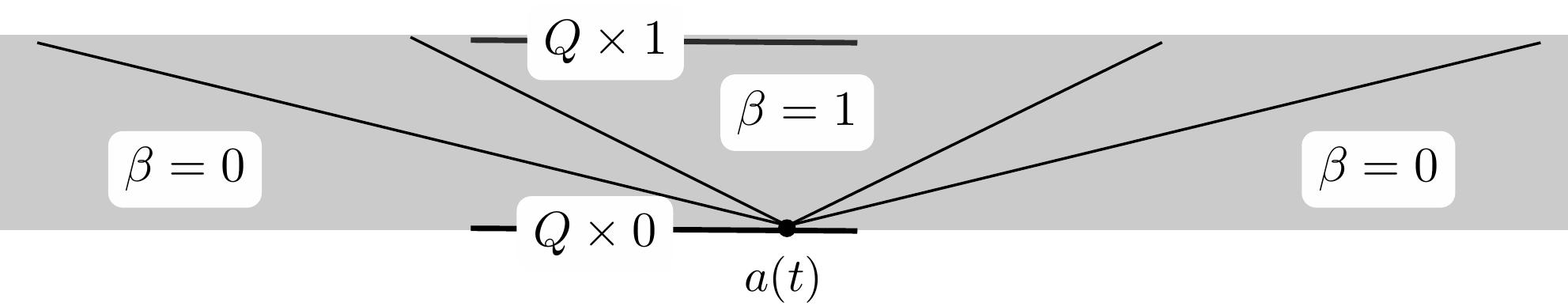}
\caption{$\beta $ is continuous except at $(t,a(t),0)$.}
\label{f.betaa}
\end{figure}

Although $\beta $ is discontinuous at $(t,a(t),0)$, the average
$$
h (t,x,z) = \beta (t,x,z)f(t,a(t)) + (1-\beta (t,x,z))f(t,x)
$$
is continuous on the whole slab.  The set $N_{z} = \{(t,x,z) : \abs{x-a(t)} < 2z\}$ is an open tubular neighborhood of   the curve $(t,a(t),z)$.  On $N_{z}$, $\beta  = 1$.  

We claim that for $0 < z \leq  1$, $a(t)$ is the unique solution of 
$$
x^{\prime} = h (t, x, z)  \quad x(0) = p \ .
$$
Let $x(t)$ be any solution of this equation.  It 
starts out in $N_{z}$, where   $\beta  = 1$ implies   $h (t,x,z) = f(t,a(t))$, a function that does not depend on $x$.  Thus, for small $t$ the solution is unique and   given by integration 
$$
x(t) = p + \int_{0}^{t} f(s,a(s)) \, ds   \ ,
$$
which is the same as $a(t)$.  Thus $x(t) = a(t)$ for small $t$.  Since $a(t)$   always lies in $N_{z}$, equality continues and we get uniqueness.  

In terms of funnels, this shows that 
$$
K_{1}(0,p\times [0,1],h ) = A \times 0 \;\;\cup \;\; c \times [0,1] . 
$$
The same construction on the top slab gives
$$
K_{1}(0, q \times [2,3], h ) = B\times 3 \;\; \cup \;\; c \times [2,3] \ .
$$
We fill in the middle slab by linear interpolation.    For $1 \leq  z \leq  2$ we set
$$
h (t,x,z) = (2-z)h (t,x,z=1) + (z-1) h (t,x,z=2)
$$
On the cube $Q \times [1,2]$, $h $ does not depend on $x$.  It is
$$
h (t,x,z) = (2-z)h (t,a(t),z=1) + (z-1) h (t,b(t),z=2) \ .
$$
Both curves $a(t)$ and $b(t)$ stay interior to the unit cube $Q$, and so does their convex combination 
$$
c(t) = (2-z)a(t) + (z-1) b(t) \ .
$$
Then $h (t,c(t),z) = (2-z)h (t,a(t),z=1) + (z-1) h (t,b(t),z=2)$ because $h(t,c(t),z)$ does not depend on $c(t)$, so 
\begin{equation*}
\begin{split}
c(t) &= (2-z)\Big(p + \int_{0}^{t} h (s,a(s),z=1)\, ds\Big) 
\\
&+ (z-1) \Big(q + \int_{0}^{t} h (s,b(s),z=2)\, ds\Big)
\\
&= (2-z)p + (z-1)q + \int_{0}^{t} h (s, c(s),z) \, ds
\end{split}
\end{equation*}
is the unique $h $-solution starting at the point $c(0) = (0,(2-z)p + (z-1)q,z)$ on the middle segment of $L_{0}$.  Since $a(1) = c = b(1)$,  we have   $c(1) = c$. In the middle slab  the trajectories through the broken segment   $L_{0}$  end at the vertical  segment $L$.  

Finally, we extend $h $ above and below $\mathbb{R}^{m+1} \times  [0,3]$ to give it compact support.  The net effect is that we get the funnel section $K$ as in Figure~\ref{f.hsection}, and Theorem~\ref{t.proj} completes the proof.
 \end{proof}


\begin{Rmk}
By induction Theorem~\ref{t.union} applies to finite unions.  But in  general, it is not true that a countable union of funnel section must be a funnel section.  For example   we can decompose the closed outward spiral into countably many arcs but it is not a funnel section.
\end{Rmk}





\section{Diffeotopies and Funnels}
\label{s.diffeotopies}

A \textbf{diffeotopy} is a smooth curve $\varphi (t)$  in the space of diffeomorphisms, starting  at the identity map when $t=0$.  We often write $\varphi (t)(x) = \varphi (t,x)$.  


Diffeotopies are generated by time-dependent ODEs and vice versa.  More precisely, if  $x(t, t_{0}, x_{0})$ solves the smooth time dependent ODE 
 $$
 x^{\prime} = f(t, x)  \qquad x(t_{0}) = x_{0}
 $$
 then
$\varphi (t,x_{0}) = x(t, 0, x_{0})$ is a diffeotopy.  Conversely, if $\varphi $ is a diffeotopy then $\varphi $ solves the ODE above with  $f(t,x) = \varphi ^{\prime}(t)(\varphi (t)^{-1}( x))$.  A diffeotopy $\varphi  $ defined on $[0,c)$ is said to have \textbf{bounded speed} if $\abs{\varphi ^{\prime}(t,x)}$ is uniformly bounded.  In this case
$$
\phi (x) = \lim_{t\rightarrow c} \varphi (t,x)
$$
exists and is continuous, although it need not   be a diffeomorphism.  Also, if $\varphi (t,x)$ is independent from $t$ for all $t \geq  c$ then the map $\phi $ defined by $x \mapsto \varphi (c,x) = \phi (x)$ is the \textbf{transfer map} of the diffeotopy.  It is the ultimate effect of the diffeotopy  on $x$.

 \begin{Thm}
 \label{t.diffeotopy}
Suppose that $A$ is a funnel section and there is a diffeotopy $\varphi $ of bounded speed on $[0,1)$ whose time one map   carries $A$ onto $B$.  Then $B$ is a funnel section.
 \end{Thm}
 \begin{proof}[\bf Proof]
 By assumption there is an ODE
$$
x^{\prime}=f(t,x)  \quad x(0) = p
$$
whose funnel has cross-section $A$ at time $1$.  By Proposition~2.4 of \cite{Pugh} we may assume that $f$ has compact support in $[0,1] \times  \mathbb{R}^{m}$.  The diffeotopy $\varphi $ gives a second  ODE, 
$$
x^{\prime} = g(t,x)
$$
whose solutions give a funnel from $0 \times  A$ to $1 \times  B$.  Since $\varphi $ has bounded speed, if we   reparameterize time as  $\tau (t) = t^{2}(2-t)^{2}$ then  the diffeotopy $\psi (t,y) = \varphi (\tau (t),y)$ has
$$
\abs{\psi ^{\prime}(t,y)} = \abs{\varphi ^{\prime}(\tau (t), y)}\tau ^{\prime}(t)  \leq M\tau ^{\prime}(t)
$$
where $M$ is the maximum speed of $\varphi $.  That is, $\psi $ is generated by an ODE which converges to zero as $t \rightarrow 1$.  
This lets us assume   $g : \mathbb{R}^{m+1} \rightarrow  \mathbb{R}^{m}$ is continuous  and has compact support in $[0,1] \times  \mathbb{R}^{m}$.   Set
$$
h(t,x) = f(t, x) + g(t-1,x) \ .
$$
Then $K_{2}(0,p,h) = B$.  
\end{proof}
%

\begin{proof}[\bf Proof of Theorems~\ref{t.pJordan} and \ref{t.pJordanL}]
Since infinite resistance implies nowhere smoothly pierceable, it suffices to prove Theorem~\ref{t.pJordanL}: \emph{there exist Jordan curves of infinite resistance, some of which are funnel sections.}   The first assertion is proved in Section~\ref{s.nowhere}.  It remains to prove that some Jordan curves of infinite resistance are funnel sections.   
By Theorem~\ref{t.diffeotopy} it is enough  to find a diffeotopy of bounded speed from the circle to some Jordan curve   of infinite resistance.  For the circle is a funnel section.

We fix a sequence $(t_{n})$ such that $0=t_{0} < t_{1} < \dots $ and $t_{n} \rightarrow  1$ as $n \rightarrow \infty$.  Then we  construct a sequence of smooth diffeotopies $\varphi _{n}$ on $S^{2}$ such that $\varphi _{n}$ is supported in the time interval $(t_{n}, t_{n+1})$.     The transfer map $\phi _{n}$ is a diffeomorphism $S^{2} \rightarrow  S^{2}$ and we   arrange things so that the composed transfer map $  \phi _{n} \circ  \cdots \circ  \phi _{0}$ converges to a homeomorphism sending the equator of $S^{2}$ to a Jordan curve of infinite resistance.  
The construction is by induction.  

First we make a general construction for any fixed smooth Jordan curve $J \subset S^{2}$ that separates the poles and has $r(J) > \alpha $.   Lemma~\ref{l.lsc} provides   a   $\delta  = \delta (J)$ such that if $J^{\prime}$ is a  Jordan curve   that separates the poles and has $d_{H}(J, J^{\prime}) < \delta $ then $r(J^{\prime}) > \alpha $.  Lemmas~\ref{l.zipper} and \ref{l.zippers} imply that there is a diffeotopy $\varphi $ such that
 \begin{itemize}

\item 
$\varphi $ is supported in a thin tubular neighborhood of $J$ and $\abs{\varphi ^{\prime}}$ is arbitrarily small.
\item 
The transfer map $\phi $ and its inverse $\phi ^{-1}$ are arbitrarily close to the identity map in the $C^{0}$ sense.
\item 
The smooth Jordan curve $\phi (J)$ separates the poles and $d_{H}(J, \phi (J)) < \delta (J)/2$.  
\item  
$r(\phi (J)) > \alpha  + 1$.

\end{itemize}

Start with $J_{0}$ equal to the equator of $S^{2}$.  It has $r(J_{0}) = \pi $.  The identity diffeotopy $\varphi _{0}$ has an identity transfer map $\phi _{0}$ and it sends the equator $J_{0}$ to itself, i.e., $J_{1} = \phi _{0}(J_{0}) = J_{0}$.  Trivially, $r(J_{1}) > 1$.

Next, applying the preceding construction to $J_{1}$, we find   a diffeotopy $\varphi _{1}$  supported on $(t_{1}, t_{2}) \times  N$ where $N$ is an equatorial band, such that the    transfer map $\phi _{1}$  carries $J_{1}$ to a smooth Jordan curve $J_{2} = \phi_{1} (J_{1})$.  Since the poles stay fixed during the diffeotopy, $J_{2}$ separates them. The construction permits 
\begin{itemize}

\item[(a$_{1}$)]
$\abs{\varphi _{1}^{\prime}} < 1$.
\item[(b$_{1}$)]
 $\abs{\phi _{1}(x) - x} < 1/2$ and $\abs{\phi _{1}^{-1}(x) - x} < 1/2$ for all $x \in S^{2}$.

\item[(c$_{1}$)]
$r(J_{2}) > 2$.  (This is trivial since the resistance is always $\geq  \pi $.)
\end{itemize}  

Inductively, assume we have defined $\varphi _{n-1}$ with time support in $(t_{n-1}, t_{n})$ and transfer map $\phi _{n-1}$.  Then $J_{n} = \phi _{n-1}(J_{n-1})$ is defined.  Working in a thin tubular neighborhood of $J_{n}$ we construct a diffeotopy $\varphi _{n}$ such that 
\begin{itemize}

\item[(a$_{n}$)]
 $\abs{\varphi _{n}^{\prime} }< 1/n$.

\item[(b$_{n}$)]
For all $x \in S^{2}$, the composed transfer maps and their inverses satisfy 
\begin{equation*}
\begin{split}
\abs{\phi _{n} \circ \phi _{n-1} \circ \cdots \circ \phi _{1}(x) - \phi _{n-1} \circ \cdots \circ \phi _{1}(x)} < \frac{1}{2^{n}}
\\
\abs{\phi _{1} ^{-1} \circ \cdots  \phi _{n-1}^{-1} \circ  \phi _{n}^{-1} (x) -   \phi _{1} ^{-1} \circ \cdots  \phi _{n-1}^{-1}(x) } < \frac{1}{2^{n}}
\end{split}
\end{equation*}

\item[(c$_{n}$)]

If $J_{n+1} = \phi _{n}(J_{n})$ then 
$$
d_{H}(J_{n}, J_{n+1}) <  \min \Big(\frac{\delta (J_{1})}{2^{n}}, \dots , \frac{\delta (J_{n})}{2 } \Big) \ .
$$
\end{itemize}

By (a$_{n}$) the diffeotopy $\bigcup \varphi _{n}$ has bounded speed on $0 \leq  t < 1$.  

 Consider   $\Phi _{n}= \phi _{n} \circ  \cdots \circ  \phi _{1}$.  By (b$_{n}$) we have
 $$
 \norm{\Phi _{n} - \Phi _{n-1}} < \frac{1}{2^{n}} \quad \textrm{and} \quad \norm{\Phi _{n}^{-1} - \Phi _{n-1}^{-1}} < \frac{1}{2^{n}} \ ,
 $$
so the sequence  $(\Phi _{n})$   is Cauchy in the space of homeomorphisms of $S^{2}$, and it converges uniformly to a homeomorphism $\Phi $ of $S^{2}$.  Let $J = \Phi (J_{0})$.  It is a Jordan curve in $S^{2}$.  Since the poles stay fixed under the diffeotopies, $J$ separates them.  

By  (c$_{n}$),
$$
d_{H}(J_{n},J) \leq \sum_{k=n}^{\infty} d_{H}(J_{k},J_{k+1}) 
< 
\sum_{k=1}^{\infty} \frac{\delta (J_{n})}{2^{k}}  =  \delta (J_{n}) \ ,
$$
which implies that $r(J) > n$.  Hence $r(J) = \infty$.  Since we arrived at $J$ by a funnel from $p$ to the equator, followed by a funnel from the equator to $J$,  $J$ is a funnel section.  
\end{proof}

\begin{Rmk}
The diffeotopy produced above ends with a homeomorphism of the  sphere  to itself and is therefore reversible. The reverse funnel from the Jordan curve leads back to the equator, $K_{0}(1\times J) = E$, and in the terminology of \cite{Pugh} we have a ``funnel cobordism'' between $E$ and $J$ .
\end{Rmk}

\begin{Rmk}

We do not know whether for every planar Jordan curve $J$ there is a diffeotopy of bounded speed that starts at the equator and ends at $J$.  If we did then we would know that every planar Jordan curve is a funnel section.
\end{Rmk}

The    proof of Theorem~\ref{t.pJordanL} above establishes the following approximation result.  
\begin{Thm}
\label{t.approximation}
A smooth Jordan curve can be approximated by other smooth Jordan curves having arbitrarily large resistance.  That is, if $h : S^{1} \rightarrow  \mathbb{C}$ sends $S^{1}$ diffeomorphically onto a Jordan curve $J$ and $\delta , L > 0$ are given, then there is an $h_{1} : S^{1} \rightarrow  \mathbb{C}$ sending $S^{1}$ diffeomorphically onto a smooth Jordan curve $J_{1}$ such that $\norm{h-h_{1}} < \delta  $ and $r(J_{1}) > L$.
\end{Thm}

\section{An Alexander Horned Sphere is a funnel section}

Consider instead of a Jordan curve, an Alexander Horned sphere $A$ \cite{Alexander}.  We claim   there is a diffeotopy $\varphi $ on $[0,1)\times \mathbb{R}^{3}$ of bounded speed that starts at the sphere $S^{2}$ and ends at $A$.  Theorem~\ref{t.diffeotopy} and the fact that $S^{2}$ is a funnel section imply that $A$ is a funnel section.    The time one map of the diffeotopy is continuous but it cannot be   a homeomorphism because the complementary domains of $S^{2}$ and $A$ are not homeomorphic.

The word ``an'' indicates that, as with Jordan curves, we do not know that every Alexander Horned Sphere is a funnel section, only that some of them are.  
Theorem~\ref{t.diffeotopy} is what may have been intended on page 283 of \cite{Pugh} by the phrase ``Using the methods of Section~4, it also follows that the usual Alexander Horned Sphere is a funnel section and so is the [closure of the] set it bounds.''

\begin{figure}[htbp]
\centering
\includegraphics[scale=.70]{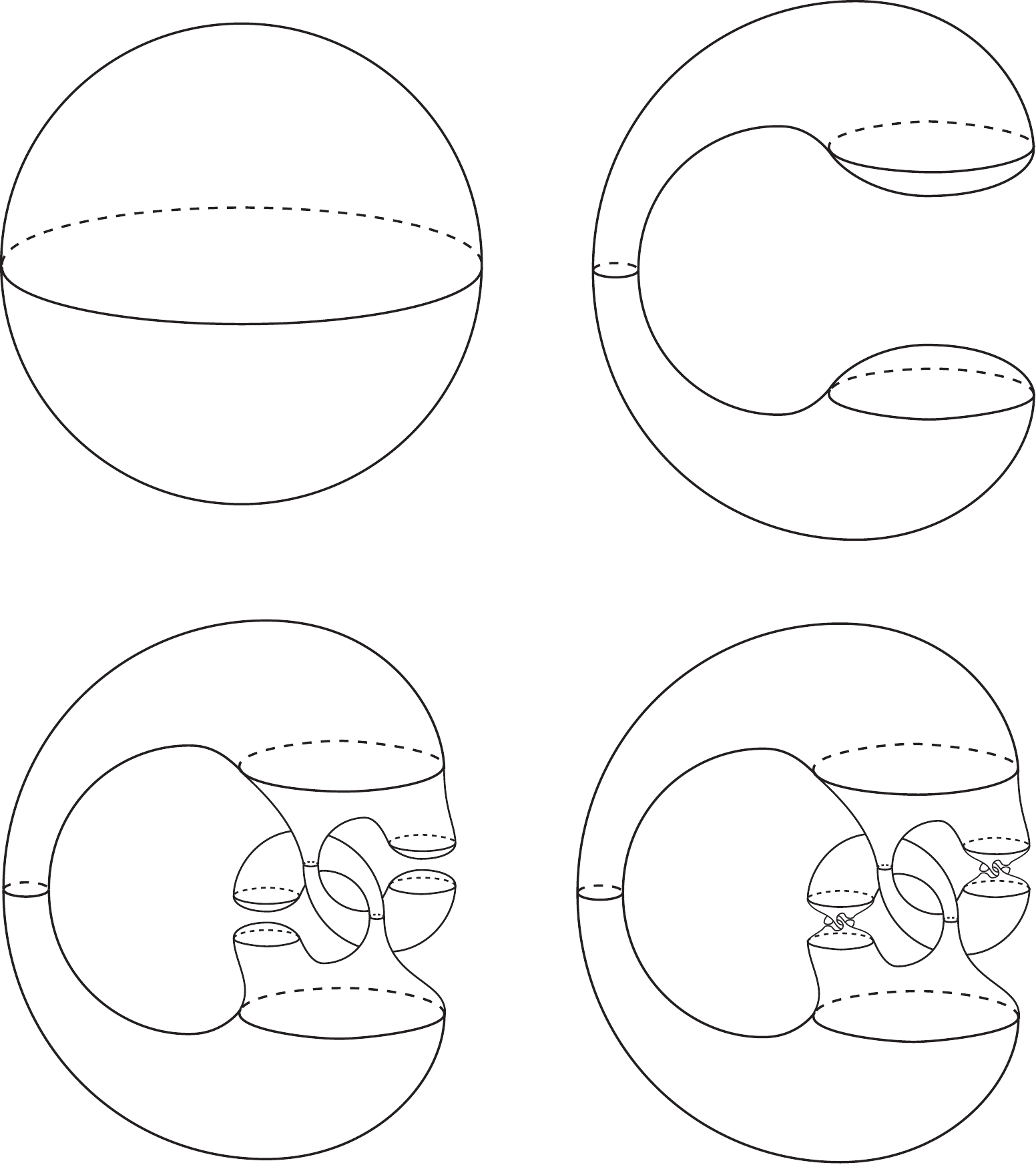}
\caption{The sphere's image at  time $0, 1/2, 3/4$ and $7/8$   of the diffeotopy.}
\label{horned sphere}
\end{figure}



As a preliminary step we easily construct a diffeotopy $\varphi _{0}$ supported in the time interval $(0,t_{1})$ that bends   the sphere into a banana shape so that the   polar caps at the north and south poles become supported on a pair of parallel discs of diameter $1$ and distance $1$ apart.  This  is shown in the second part of Figure~\ref{horned sphere}.   The resulting smooth sphere is $ S_{1} = \phi _{0}(S^{2})$ where $\phi _{0}$ is the transfer map of $\varphi _{0}$.       

Next we define a diffeotopy $\varphi _{1}$ on the time interval  $(1/2, 3/4)$ that fixes all points of $S_{1}$ in the complement of the two parallel caps and moves four disjoint discs in the caps to the four smaller caps shown in the third part of Figure~\ref{horned sphere}. The four discs have diameter $1/4$; the diffeotopy $\varphi _{1}$ moves them  to parallel caps of diameter $1/4$ and distance $1/4$ apart.     The resulting smooth sphere is $S_{2} = \phi  _{1}(S_{1})$.   

At the $n^{\textrm{th}}$   stage, we develop $2^{n-1}$  independent banana shapes where   the spatial dimensions are reduced by the factor $1/4 $ from the spatial dimensions at the previous stage, while the time interval is reduced by the factor $1/2$.  This is done merely by copying and scaling the   diffeotopy $\varphi _{1}$.  Since the spatial reduction dominates the time reduction the   speed of the combined diffeotopy $ \varphi = \bigcup \varphi _{n}$ tends to zero as $t\rightarrow 1$.    
Hence the whole diffeotopy on $[0, 1)$ starts at $S^{2}$, has bounded speed, and limits to our Alexander horned sphere as $t\rightarrow 1$.
As stated at the outset, since $S^{2}$ is a funnel section, so is $A$.  

The same construction done with the roles of inside and outside reversed shows that an Alexander Horned Sphere with inward curling horns is also a funnel section, as is an Alexander Horned Ball.

\section{A complete metric on the space of Jordan curves}
\label{s.complete}

The space of homeomorphisms of a compact metric space to itself has a natural complete metric, but the same does not seem to be true for the space of topological embeddings of a compact metric space into another metric space.  In the case of planar Jordan curves, we  use the Riemann mapping theorem to get such a metric. Many thanks to Andy Hammerlindl and Bill Thurston for elegant suggestions regarding the construction of such a metric.  

As above, let $\mathcal{J}_{0}$ denote the set of Jordan curves in $\widehat{\mathbb{C}}$ that separate the poles.  Given $J \in \mathcal{J}_{0}$, Caratheodory's Theorem supplies unique conformal bijections  
$$
R : \mathbb{D}  \rightarrow  \Omega  \qquad \widetilde{R} : \mathbb{D}  \rightarrow \widetilde{\Omega }
$$
such that 
\begin{itemize}

\item

$\Omega $ and $\widetilde{\Omega }$ are the connected components of $\widehat{\mathbb{C}} \setminus J$ containing the south pole and north pole respectively.

\item
$R(0)$ is the south pole and $\widetilde{R}(0)$ is the north pole.

\item
If $\pi $ denotes stereographic projection then $(\pi \circ R)^{\prime}(0) $ is real and positive.
\item
If $\alpha $ denotes inversion $z \mapsto  1/z$ then $(\pi \circ \alpha \circ \widetilde{R})^{\prime}(0) $ is real and positive.
\end{itemize}
We refer to $R$ and $\widetilde{R}$ as the \textbf{canonical Riemann maps} corresponding to $J \in \mathcal{J}_{0}$.

\begin{Defn}
For $J_{1}, J_{2} \in \mathcal{J}_{0}$, set
$$
d(J_{1}, J_{2}) = \norm{R_{1} - R_{2}} + \norm{\widetilde{R}_{1} - \widetilde{R}_{2}} \ ,
$$
where $R_{1}, \widetilde{R}_{1}$ and $R_{2}, \widetilde{R}_{2}$ are the canonical Riemann maps corresponding to $J_{1}$ and $J_{2}$.  (Recall that $\norm{F-G}$ is the $C^{0}$ distance between $F$ and $G$.) It is clear that $d$ is a metric on $\mathcal{J}_{0}$, and we call it the 
  \textbf{welding metric}.  For it deals with   pairs $R \sqcup \widetilde{R}$ welded by   $R^{-1} \circ  \widetilde{R}|_{\partial \mathbb{D} }$.
\end{Defn}

\begin{Thm}
\label{t. }
The welding metric is complete.
\end{Thm}

\begin{proof}[\bf Proof]
To show that $d$ is complete, let   $(J_{n})$ be a Cauchy sequence in $\mathcal{J}_{0}$.  The Riemann maps   $R_{n}$ and $\widetilde{R}_{n}$ corresponding to $J_{n}$ converge uniformly to continuous maps $R$ and $\widetilde{R}$  from the closed disc into $\widehat{\mathbb{C}}$.  Uniform convergence and $R_{n}(\partial \mathbb{D} ) = J_{n} = \widetilde{R}_{n}(\partial \mathbb{D} )$ imply that 
$$
R(\partial \mathbb{D} ) = J = \widetilde{R}(\partial \mathbb{D} ) \ .
$$
This shows that $J$ is a curve, but we don't yet know     it's a Jordan curve, nor that $J_{n}$ converges to it with respect to $d$.  


 By Hurwitz' Theorem, the restriction of $\widetilde{R}$ to the open disc is either constant or a \textbf{holeomorphism}  --  a holomorphic homeomorphism.  But if $\widetilde{R}$ is constant on the open disc then by continuity and the fact that each $\widetilde{R}_{n}$ sends the origin to the north pole, $\widetilde{R}(\overline{\mathbb{D}} )$ \emph{is} the north pole.  This implies that for $n$ large, the Jordan curve $J_{n}$ approximates the north pole, and its southern complementary region $\Omega _{n}$ approximates $ \pi ^{-1}(\mathbb{C})  $, the sphere minus the north pole.  Consequently, $\pi \circ R_{n} : \mathbb{D} \rightarrow \mathbb{C}$ converges to a holeomorphism $\mathbb{D} \rightarrow \mathbb{C} $, a contradiction to Liouville's Theorem.  Therefore   $\widetilde{R}$ sends $\mathbb{D} $ holeomorphically onto a region $\widetilde{\Omega} \subset \widehat{\mathbb{C}}$ containing the north pole, and similarly, $R$ sends $\mathbb{D} $ holeomorphically onto a region $\Omega \subset  \widehat{\mathbb{C}}$ containing the south pole.  
 
 We claim that $R$ and $\widetilde{R}$ are injective. Suppose not: there exist  $p, q \in \overline{\mathbb{D} }$ with $p \not= q$ and $R(p) = R(q)$.   We first show  that $p, q \in \partial \mathbb{D} $.  We know that $R$ sends the open disc $\mathbb{D} $ holeomorphically onto the region $\Omega $, so at least one of $p$ and $q$ belongs to $\partial \mathbb{D} $.  If $p \in \mathbb{D} $ and $q \in \partial \mathbb{D} $ then there are points $q^{\prime} \in \mathbb{D} $ near $q$ sent to points near $R(p)$, contradicting the fact that all points near $R(p)$ are $R$-images of points near $p$.  
 
%
%
%
%
%
%
%
%

Thus, if $R(p) = R(q)$ and $p \not=  q$ then $p, q \in \partial \mathbb{D} $.  Consider the radial segments $P = [0,p]$ and $Q = [0,q]$ in $\overline{\mathbb{D} }$.   Their $R$-images   are arcs from  the south pole to $z = R(p) = R(q)$.  Except for the south pole and $z$, the arcs are disjoint,  so their union is a Jordan curve $H$ contained in $\Omega \cup \{z\}$.    The south pole belongs to $H$ but the north pole does not.    Let $U$ be the complementary region of $H$ contained in $\Omega $.  (Equivalently, it is the complementary region   that does not contain the north pole.) Let $V =  R^{-1}(U) $.  It is contained in one of the two sectors, say $S$,   in $\mathbb{D} $ bounded by the radial segments  $P$, $Q$, and a  circular arc $A$ joining $p$ to $q$.   See Figure~\ref{f.HandL}.
\begin{figure}[htbp]
\centering
\includegraphics[scale=.65]{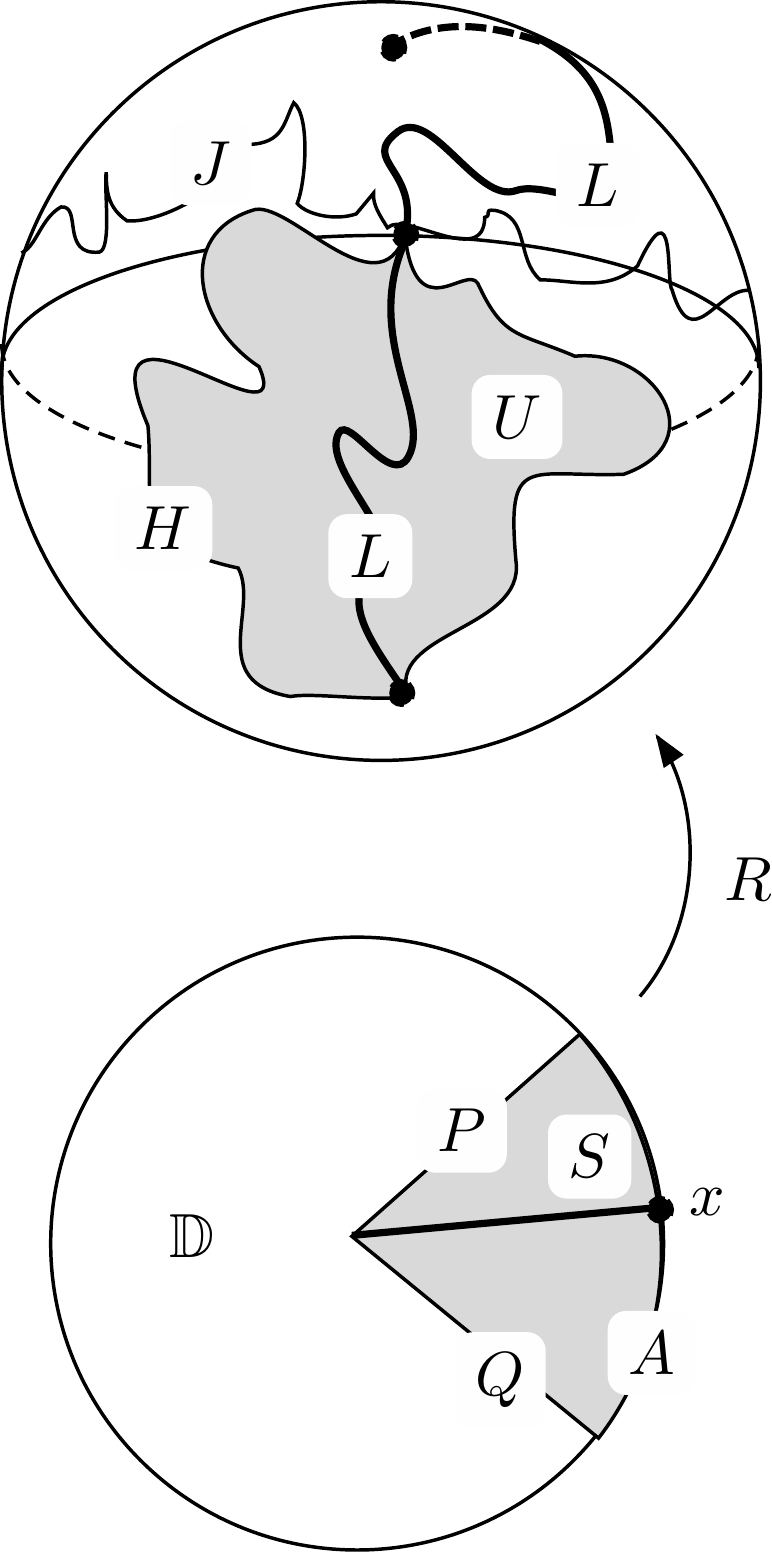}
\caption{The Riemann map $R$ sends $\mathbb{D} $ to the southern region bounded by $J$.  It also sends  $V \subset S$ to $U$,    $P \cup Q$ to the Jordan curve $H$, and $[0,x]$ to the southern part   of the arc $L$.}
\label{f.HandL}
\end{figure}

Fix any     $x\in A$.  We have $R(x) \in J$, so there is at least one $\widetilde{x} \in \partial \mathbb{D} $ with $\widetilde{R}(\widetilde{x}) = R(x)$.  Then $L  = R([0,x]) \cup \widetilde{R}([0,\widetilde{x}])$ is an arc in $S^{2}$ from the south pole to the north pole, and it  meets $J$ only at $R(x)$.   Since $L$ starts inside $H$ and ends outside $H$, it crosses $H$ somewhere, say at $z^{\prime}$.  Injectivity of $R$ implies that $z^{\prime}$  does not lie in the open edges of $H$, $R((0,p) \cup (0,q))$, so $z^{\prime}=z$, and   $R(A) = \{z\}$.  This is a contradiction to the general fact that a holeomorphism from a Jordan domain (such as $S$) to a Jordan domain (such as $U$) is always a homeomorphism from the boundary  of the first  to the boundary of the second.    

Thus, $R$ and $\widetilde{R}$ are   injective.  They send $\mathbb{D} $ to the disjoint regions $\Omega $ and $\widetilde{\Omega }$, and they  both send $\partial \mathbb{D} $ onto $J$.  Thus  $J$ is a Jordan curve.  Its   complementary regions,   $\Omega $ and $\widetilde{\Omega }$, contain the south and north poles respectively, so  $J \in \mathcal{J}_{0}$ and $J_{n} \rightarrow  J$.
\end{proof}

\begin{Rmk}
The set $\mathcal{J}$ of all Jordan curves $J \subset  S^{2}$ receives a natural \textbf{welding topology} as well.  As remarked at the end of Section~\ref{s.nowhere}, there is nothing special about the north and south pole.  For any pair of distinct points $p, q \in S^{2}$, the set of Jordan curves separating them, say $\mathcal{J}_{pq}$, has a metric topology given from the pull-back of the welding metric on $\mathcal{J}_{0}$ under a Mobius transformation of $\widehat{\mathbb{C}}$ sending the poles to $p$ and $q$.  The welding topology is locally unchanged if $p$ and $q$ are varied slightly.  Taking $p,q$ as distinct points in a countable dense subset of $\widehat{\mathbb{C}}$, we see that $\mathcal{J}$ is locally a complete metric space.  Hence $\mathcal{J}$ is a Baire space, so it makes sense to speak of the generic Jordan curve, and to ask what properties it has.
\end{Rmk}

\begin{Rmk}
There are other metrics that give the same topology to the space of Jordan curves, but the welding metric has the advantage of being complete.  As two examples, consider  
\begin{equation*}
\begin{split}
d_{1}(J_{1}, J_{2}) & = \norm{R_{1} - R_{2}}
\\
d_{2}(J_{1}, J_{2}) & = \inf   \norm{h_{1} - h_{2}} 
\end{split}
\end{equation*}
where the infimum is taken over all pairs of homeomorphisms $h_{1} : S^{1} \rightarrow  J_{1}$, $h_{2} : S^{1} \rightarrow  J_{2}$.  By   Rad\'o's Theorem (below) these metrics are topologically equivalent to the welding metric but are not metrically comparable to it.  
\end{Rmk}

\section{Generic Jordan curves}

\begin{Thm}
\label{t.generic}
The generic    Jordan curve is nowhere pierceable by paths of finite length.
\end{Thm}

We will use the following result of Rad\'o \cite{Rado} to prove this.  See also Chapter~2 of Pommerenke's book   \cite{Pommerenke}.

\begin{Rado}
\label{t.Rado}
Suppose that $J$ is a planar Jordan curve  enclosing the origin and $h : S^{1} \rightarrow  J$ is a homeomorphism.  Given $\epsilon  > 0$ there is a $\delta  > 0$ such that if $J_{1}$ is a Jordan curve and $h_{1} : S^{1} \rightarrow  J_{1}$ is a homeomorphism with $\norm{h-h_{1}} < \delta $ then $J_{1}$ encloses the origin and $\norm{R-R_{1}} < \epsilon $ where $R$ and $ R_{1} $ are the canonical Riemann maps for $J$ and $J_{1}$.    
\end{Rado}

\begin{proof}[\bf Proof of Theorem~\ref{t.generic}]  


Consider the set $\mathcal{J}_{0}(L) = \{ J \in \mathcal{J}_{0} : r(J) > L \}$.  We claim it is open and dense in $\mathcal{J}_{0}$ with respect to the welding metric $d$ defined in Section~\ref{s.complete}.  Openness follows from Lemma~\ref{l.lsc}, lower semicontinuity of the resistance function with respect to the Hausdorff metric topology on $\mathcal{J}_{0}$, and the fact that the latter topology is coarser than the welding topology. 

To check density, let $J \in \mathcal{J}_{0}$ and $\epsilon  > 0$ be given.  Let $R $ be the canonical Riemann map fixing the origin and  sending $\mathbb{D} $ onto the planar region $\Omega $ bounded by $J$.   Then $h = R|_{\partial \mathbb{D} } : S^{1} \rightarrow  J$ is a homeomorphism and Rad\'o's Theorem supplies a $\delta  > 0$ such that if $h_{1} : S^{1} \rightarrow  J_{1}$ is a homeomorphism and $\norm{h-h_{1}} < \delta $ then $\norm{R-R_{1}} < \epsilon /2 $.  Rad\'o's Theorem   applies equally  to the outer complementary region of $J$, and we infer that $\norm{h-h_{1}} < \delta $ implies $\norm{\widetilde{R}-\widetilde{R}_{1}} < \epsilon /2 $.

Continuity of $R$ implies that for   $\rho < 1$, the   map  
$$
h_{\rho }  : e^{i\theta } \mapsto  R(\rho e^{i\theta })
$$
is a diffeomorphism from $S^{1}$ onto the smooth Jordan curve $J_{\rho } $, which is the $R$-image of the circle of radius $\rho $.  If $\rho $ is near $1$ then    $\norm{h-h_{\rho }} < \delta /2$.      Since $J_{\rho }$ is smooth, Theorem~\ref{t.approximation} implies there is a Jordan curve $J_{1}$  and an $h_{1} : S^{1} \rightarrow \mathbb{C}$ sending $S^{1}$ diffeomorphically onto $J_{1}$ such that $\norm{h_{\rho } - h_{1}} < \delta /2$ and $r(J_{1}) > L$.  Thus $\norm{h-h_{1}} < \delta $, which implies 
$$
d(J , J_{1}) = \norm{R  - R_{1}} + \norm{\widetilde{R} - \widetilde{R}_{1}}  < \epsilon \ , 
$$
and confirms density of  $
\mathcal{J}_{0}(L)$   in $\mathcal{J}_{0}$.  The countable intersection $\bigcap_{L \in \mathbb{N}} \mathcal{J}_{0}(L)$ is residual, so the generic Jordan curve has infinite resistance:   it is nowhere pierceable by paths of finite length.  Since $\mathcal{J}$ is a Baire space, locally homemorphic to $\mathcal{J}_{0}$, the same holds for the generic  $ J \in \mathcal{J}$.
\end{proof}

\ \

Charles Pugh, Department of Mathematics, Northwestern University.

\emph{pugh@math.northwestern.edu}

\ \

Conan Wu, Department of Mathematics, Princeton University.

\emph{shuyunwu@math.princeton.edu}

\end{document}